\documentclass[12pt]{article}

\usepackage{titlesec}
   \titleformat{\chapter}
      {\normalfont\large}{Chapter \thechapter:}{1em}{}

\usepackage{times}              
\usepackage{epsf}
\usepackage{amsfonts}
\usepackage{fancyhdr}
\usepackage{amsthm}
\usepackage{amsmath}
\usepackage{graphicx}
\usepackage{url}

\usepackage{graphicx}
\usepackage{cite}
\usepackage{lscape}
\usepackage{indentfirst}
\usepackage{latexsym}
\usepackage{multirow}
\usepackage{tabls}
\usepackage{wrapfig}
\usepackage{slashbox}
\usepackage{longtable}
\usepackage{supertabular}
\usepackage{subfigure}

\title{On conformally flat circle bundles over surfaces}
\author{Son Lam Ho\\
  \small Mathematics Research Unit\\
  \small University of Luxembourg\\
  \small Luxembourg
}

\begin{document}

\newtheorem{definition}{Definition}[section]
\newtheorem{theorem}[definition]{Theorem}
\newtheorem{lemma}[definition]{Lemma}
\newtheorem{question}[definition]{Question}
\newtheorem{corollary}[definition]{Corollary}
\newtheorem{proposition}[definition]{Proposition}
\newtheorem{claim}[definition]{Claim}
\newtheorem{remark}[definition]{Remark}
\newtheorem{conjecture}[definition]{Conjecture}
\newtheorem*{maintheorem}{Main Theorem}
\theoremstyle{remark}
\newtheorem{example}{Example}
\newcommand{\Mob}{\textnormal{M\"ob}}
\newcommand{\Isom}{\textnormal{Isom}}








\maketitle

\section{Introduction}
\label{secIntro}
Let $E$ be the total space of a circle bundle over a closed surface $\Sigma_g$ of genus $g$ with $g\ge 2$. If we fix $g$, the topological type of $E$ is classified by its Euler number, denoted by $e(E)$.  Gromov, Lawson, and Thurston \cite{GLT}, Kuiper \cite{Kuiper88}, and Kapovich \cite{Kapovich89Flat} constructed examples of flat conformal $(SO(4,1), S^3)$ structures on circle bundles $E$ with non-zero Euler number. These examples are constructed using fundamental domains in $S^3$ that are bounded by a ``necklace" of 2-spheres which are arranged along an unknotted embedding of a circle. It's interesting that the arrangement of these spheres determines the topology of $E$.


All constructed examples of flat conformal structures on $E$ satisfy the inequality
\begin{equation}\label{MW}
|e(E)|\le |\chi(\Sigma_g)| = 2g - 2.
\end{equation}
It is conjectured in \cite{GLT} that this inequality is a necessary condition for the existence of a flat conformal structure on $E$. We will refer to this as the GLT conjecture. If it is true, it would be an example of the general principle that existence of geometry on a manifold often restricts its topology.

The case when $g=1$ was established by Goldman in \cite{Goldman83Conf}. In this article we consider the case $g\ge 2$, in particular the case where the conformally flat manifold $E$ is a quotient of the domain of discontinuity of a surface group $\Gamma$. Partial results are obtained in some nice cases, one of them is: $e(E)<\frac{3}{2}n^2$ for structures with a certain type of fundamental domains of $n$ faces. These soft bounds are presented in section \ref{secEulernumber}.

Representations of fundamental groups of closed surfaces (surface groups) into a Lie group $G$ is a well-studied subject. This is especially true when $G$ is $\Isom(\mathbb{H}^2)$ and $\Isom(\mathbb{H}^3)$. Here we study the natural higher dimensional analogue: surface groups in $G=\Isom(\mathbb{H}^4)$. Just as in lower dimension cases, we want to study the space of ``nice" surface groups called quasi-Fuchsian groups. Following \cite{Kapovich07Klein} we define:
\begin{definition}
A surface group in $\Isom(\mathbb{H}^n)$ (for $n\ge 2$) is \emph{quasi-Fuchsian} if its limit set is a topological circle in $\partial_\infty (\mathbb{H}^n) = S^{n-1}$.
\end{definition}
Indeed, all known examples of conformally flat circle bundles are constructed as quotients of a domain in $S^3$ by a quasi-Fuchsian surface group which admits a finite sided fundamental domain. In non-trivial cases (where $e(E)\ne 0$,) the limit set of one of these group is a fractal topological circle in $S^3$ with torsion.

The article is organized as follows. In section 2, we present background material in flat conformal geometry. In section \ref{SecCircleBundles} we describe the combinatorial construction of a circle bundle, section \ref{SecFundomain} and \ref{SecDeform} discuss examples old and new, and an algorithm to compute the Euler number using the fundamental domain. In section \ref{secEulernumber} we show two approaches to bound the Euler number of a conformally flat circle bundle in nice cases.

\bigskip

\noindent {\bf Acknowledgements.} This article is part of the author's Ph.D. thesis completed at University of Maryland, College Park. I'd like to thank my Ph.D. advisor, professor Bill Goldman, for introducing me to this mathematical area, and for countless number of enlightening conversations. I also thank him for his patient guidance, generosity and friendship in many years. I'm grateful to professor Feng Luo and professor Karin Melnick for their helpful comments and suggestions, and to professor Jean-Marc Schlenker and University of Luxembourg for their support. I'd like to acknowledge support from U.S. National Science Foundation grants DMS 1107452, 1107263, 1107367 ``NMS: Geometric Structures and Representation Varieties'' (the GEAR Network).

\section{Preliminaries on flat conformal geometry}
\label{secPrelim}

Let us recall some basic concepts and notations. Let $(X,g)$ be a  Riemannian manifold. Isometries from $X$ to itself are maps preserve the Riemannian metric $g$. These maps form a group which we call $\textnormal{Isom}(X)$, the group of isometries on $X$. The group of orientation preserving isometries is denoted $\textnormal{Isom}^+(X)$.

Two Riemannian metrics $g,h$ on $X$ are conformally equivalent if there is a positive function $\lambda$ on $X$ such that $g_x(u,v)=\lambda(x) h_x(u,v)$ on each tangent space. A class of conformally equivalent metrics on $X$ is called a \emph{conformal structure}. Given two Riemannian manifolds $(X,g)$ and $(Y,h)$, a local diffeomorphism $X\rightarrow Y$ is called a \emph{conformal map} if the pull back metric $h^*$ on $X$ is conformally equivalent to $g$.

For the sphere $S^{n}$, the conformal maps from $S^n$ to itself form a group which we will name $\Mob(S^n)$.

We will now introduce the hyperbolic space and its conformal sphere boundary at infinity. More details can be found in \cite{Ratcliffe94Book}. For $m\ge 2$, we let $\mathbb{R}^{m,1}$ be  $\mathbb{R}^{m+1}$ with a Lorentzian metric of signature $(m,1)$. That is, in metric is given by the quadratic form $B(x)= - x_0^2 + x_1^2 + ... + x_m^2.$ Let $SO(m,1)$ be the group of $(m+1)\times (m+1)$ matrices of determinant $1$ preserving this quadratic form.

The hyperbolic $m$ space, denoted $\mathbb{H}^m$ is defined to be the level set $$H:=\{x\in\mathbb{R}^{m,1}|B(x)=-1\}\subset\mathbb{R}^{m,1}$$ with the inherited metric on each tangent space. This is the hyperboloid model of the hyperbolic space. The hyperboloid $H$ is asymptotic to the light-cone $$L:=\{x\in\mathbb{R}^{m,1}|B(x)= -x_0^2 + \sum_{i=1}^m x_i^2 = 0\}\subset\mathbb{R}^{m,1}.$$ The projectivization of $L$ identifies with $\{x\in\mathbb{R}^{m,1}|\sum_{i=1}^m x_i^2 = 1, x_0= 1\} = S^{m-1}$. This is a sphere of dimension $m-1$ and it is naturally the boundary at infinity of hyperbolic space $\mathbb{H}^{m}$. Unless otherwise noted we give $S^{m-1}$ the standard unit sphere metric inherited from $\mathbb{R}^m$.

It can be shown that $SO(m,1)$ acts on $\mathbb{H}^m$ by hyperbolic isometries, and on $S^{m-1}$ by conformal automorphisms. That is, $\Mob(S^{m-1}) = \textnormal{Isom}(\mathbb{H}^{m}) = SO(m,1)$.

An element $A\in \Mob^+(S^{m-1})$ can be classified by the dynamics of its action on $\mathbb{H}^{m}\cup S^{m-1}$:
\begin{enumerate}\itemsep0pt \parskip0pt \parsep0pt
\item Loxodromic: this is when $A$ has two fixed points (one attracting and one repelling) on $S^{m-1}$, and it leaves invariant a geodesic in $\mathbb{H}^{m}$ whose ideal end points are the fixed points of $A$. All loxodromic actions on $S^{m-1} = \mathbb{R}^{m-1} \cup \{\infty\}$ are conjugate to $x \mapsto R(\lambda x)$ for $x\in\mathbb{R}^{m-1}$, $1\ne\lambda\in\mathbb{R}$ a scalar, and $R$ a matrix in $SO({m-1})$ representing a rotation.
\item Parabolic: this is when $A$ has exactly one fixed point on $S^{m-1}$. We can think of a parabolic element as a limit of loxodromic elements when the two fixed points come together. All parabolic actions on $S^{m-1} = \mathbb{R}^{m-1} \cup\{ \infty\}$ are conjugate to $x \mapsto Rx + t$ for $R\in SO({m-1})$ a rotation matrix, and $t\in\mathbb{R}^{m-1}$ a non-zero vector of translations.
\item Elliptic: this is when $A$ has one or more fixed point(s) in $\mathbb{H}^{m}$. In this case, the differential of the action of $A$ at the fixed point $p\in\mathbb{H}^{m}$ can be represented as a matrix in $SO(m)$.
\end{enumerate}
(See \cite{Kapovich07Klein}, \cite{Kapovich01Book}, and \cite{Ratcliffe94Book} for more details on this classification.) One can easily show that loxodromic elements form an open set in $\Mob^+(S^{m-1})$ using Brouwer fixed point theorem.

The following are well known facts in lower dimensions. For $n=1$, we have $\Mob^+(S^1) = \textnormal{Isom}^+(\mathbb{H}^{2})\cong PSL(2,\mathbb{R})$ the group of projectivized $2\times 2$ real matrices of determinant 1. For $n=2$, we have $\Mob^+(\mathbb{C}P^1) = \textnormal{Isom}^+(\mathbb{H}^{3})\cong PSL(2,\mathbb{C})$ the group of projectivized $2\times 2$ complex matrices of determinant 1.

The basic objects in conformal geometry are subspheres inside $S^{m-1}$. These objects arise naturally as the boundaries of totally geodesic hyperbolic subspaces inside $\mathbb{H}^{m}$. Under actions of elements of $\Mob(S^{m-1})$, one subsphere must be mapped to another of the same dimension, and given two subspheres of the same dimension in $S^{m-1}$ there always exist conformal transformations taking one to the other. Moreover, under a stereographic projection 
$(S^{m-1} -\infty) \rightarrow \mathbb{R}^{m-1}$ these subspheres are mapped to Euclidean subspheres and in $\mathbb{R}^{m-1}$. We now reserve the words ``circle'' and ``sphere'' only for these natural geometric objects in $S^{m-1}$. From this point we restrict our attention to the case $m=4$. In all figures, $S^3$ will be presented as its stereographic projection model $\mathbb{R}^3 \cup \{\infty\}$.

\begin{definition}
A loxodromic transformation $A\in \Mob^+(S^3)$ is said to be \emph{non-rotating} if up to a conjugation there is a $\lambda\in\mathbb{R}+$ so that $A(x) = \lambda x$ for all $x\in\mathbb{R}^3 = S^3 - \infty$. Otherwise we say that $A$ is \emph{rotating}.
\end{definition}

\subsection{Further classifications of elliptic transformations}

Let $A\in\Mob^+(S^3)$ be an elliptic transformation. We say $A$ is \emph{regular elliptic} if it fixes exactly one point in $\mathbb{H}^4$, thus it has no fixed points in $S^3$. For example, if we let $S^3 = \{(z,w)\in\mathbb{C}^2 : |z|^2 + |w|^2 = 1 \}$, then the transformation $(z,w)\mapsto (e^{i\theta}z, e^{i\psi}w) $ for $\theta,\psi \ne k2\pi$ would be a regular elliptic transformation.

If $A$ is not regular elliptic, that is, it fixes two points in $\mathbb{H}^4$, then $A$ fixes the geodesic connecting those two points. So the differential of $A$ at a fixed point can be represented as 
$d_pA \sim \left(\begin{array}{cc}
1 & 0 \\
0 & R
\end{array}\right)
$
with $R\in SO(3)$. But any $SO(3)$ rotation has a fixed axis, so
$$d_pA \sim \left(\begin{array}{ccc}
1 & 0 & 0 \\
0 & 1 & 0 \\
0 & 0 & T
\end{array}\right)
$$
with $T\in SO(2)$. Thus if $A$ is not regular elliptic, then A fixes a totally geodesic plane inside $\mathbb{H}^4$ along with its ideal boundary: a circle in $S^3$. In this case, we say that $A$ is \emph{singular-elliptic}.

This is an important class of M\" obius transformations. The centralizer of almost every 1-parameter subgroup of $\Mob^+(S^3)$ contains a singular-elliptic subgroup. (1-parameter subgroups generated by certain parabolic elements may have trivial centralizer.)
Also, among all 1-parameter subgroups of $\Mob^+(S^3)$, the singular-elliptic ones have the largest centralizers, which can be shown to be isomorphic to $SO(2)\times PSL(2,\mathbb{R})$.

\subsection{M\"obius Annulus}
\label{subsecMobAnnu}
Let $C$ be a circle in $S^3$. We denote Fix$(C)$ the subgroup of $\Mob^+(S^3)$ that fixes every point on $C$. Then Fix$(C)$ contains only singular-elliptic elements, and it is isomorphic to $SO(2)$. Its action on $S^3$ is rotation around the axis $C$.

Consider Inv$^+(C)$, the subgroup of $\Mob^+(S^3)$ that leaves a circle $C$ invariant and preserve its orientation. That is, elements of Inv$^+(C)$ mapping $C$ and its orientation to itself. Clearly Fix$(C)\hookrightarrow$Inv$^+(C)$ a normal subgroup. Let $\Mob^+(C)$ be the group of orientation preserving conformal automorphisms of $C$. Then Inv$^+(C)\rightarrow \Mob^+(C) \cong \Isom^+(\mathbb{H}^2)\cong PSL(2,\mathbb{R})$ is a surjective map defined by restricting the elements of Inv$^+(C)$ to act on $C$. So we have an exact sequence:

$$0\rightarrow\textnormal{ Fix}(C)\hookrightarrow\textnormal{ Inv}^+(C) \rightarrow\Mob^+(C)\rightarrow 0.$$
In fact there is a splitting $\Mob^+(C)\hookrightarrow\textnormal{Inv}^+(C)$ and the image commutes with $\textnormal{ Fix}(C)$ . So we have $\textnormal{Inv}^+(C)\cong\textnormal{Fix}(C)\times\Mob^+(C) \cong SO(2)\times PSL(2,\mathbb{R})$.

Now let $S$ be a 2-sphere in $S^3$. We would still have a similar exact sequence like above. And since Fix$(S)$ is the trivial group in $\Mob^+(S^3)$, we get Inv$^+(S) \cong\Mob^+(S)\cong PSL(2,\mathbb{C})$. We can then think of a 2-sphere in $S^3$ as a copy of $\mathbb{C}P^1$ with $PSL(2,\mathbb{C})$ acting on it.

Let $H$ be a connected subset of a 2-sphere $S$ with boundary being a circle, that is, $H$ is a half sphere. Then Inv$^+(H)$ is a subset of Inv$^+(S) \cong PSL(2,\mathbb{C})$ that fixes a half-sphere. So Inv$^+(H)\cong PSL(2,\mathbb{R})$ the isometry group of $\mathbb{H}^2$. We can then think of $H$ as a copy of the hyperbolic plane.

An important object to study is defined below:
\begin{definition}
A \emph{M\"obius annulus}  is a 2-sphere minus two disjoint half-spheres.
\end{definition}

Let $\mathcal{A}$ be a M\"obius annulus, and denote Inv$^*(\mathcal{A})$ the subgroup of $\Mob^+(S^3)$ that leaves $\mathcal{A}$ and its two boundary components invariant. (Elements of this group will not swap the two boundary components of $\mathcal{A}$.) Let $\partial_1 \mathcal{A}$ and $\partial_2 \mathcal{A}$ be the two boundary components which are both circles. Suppose that half-sphere $H$ contains $\mathcal{A}$  with $\partial H = \partial_1 \mathcal{A}$. Then we have $\textnormal{Inv}^*(\mathcal{A})\subset\textnormal{Inv}^+(H)\cong PSL(2,\mathbb{R})$, and $\textnormal{Inv}^*(\mathcal{A})$ preserve a disk which is the complement of $\mathcal{A}$ in $H$. So 
$$\textnormal{Inv}^*(\mathcal{A})\cong SO(2)$$ a subgroup of elliptic elements in $ PSL(2,\mathbb{R})$.

Consider a M\"obius annulus in the plane given by $$\mathcal{A}_l = \{z\in\mathbb{C}\ |\  \frac{1}{e^l}<|z|<1\} \ \ \textnormal{ for } l>0$$ which is considered to be a subset of the Poincare unit disk model of the hyperbolic plane. If $0<l_1<l_2$ then $\mathcal{A}_{l_1}$ is not equivalent to $\mathcal{A}_{l_2}$. This is because conformal automorphism of the unit disk cannot change hyperbolic area.

We now claim that every M\"obius annulus in $S^3$ is conformally equivalent to $\mathcal{A}_l$ for some $l>0$. Given any M\"obius annulus $\mathcal{A}$, it is contained in a halfsphere $H$ such that $\partial H = \partial_1 \mathcal{A}$. Since $H$ can be mapped conformally to the unit disk, this map takes $\partial_1 \mathcal{A}$ to $\{|z|=1\}$ and takes the other boundary component $\partial_2 \mathcal{A}$ to some circle inside the unit disk. A composition with some element of Inv$^+(H)$ will take $\partial_2 \mathcal{A}$ to a circle centered at $0$ with radius $1/e^l$. So $\mathcal{A}$ is equivalent to $\mathcal{A}_l$ for some $l>0$. Thus we have the following:

\begin{remark}
\label{annuliInvariant}
There is a one parameter invariant given by $l>0$ for M\"obius annulus $\mathcal{A}$. We denote this by $\textnormal{mod}(\mathcal{A})$ and we call it the \emph{modulus} of $A$.
\end{remark}

\begin{figure}[h]
\centering
\includegraphics[width=9cm]{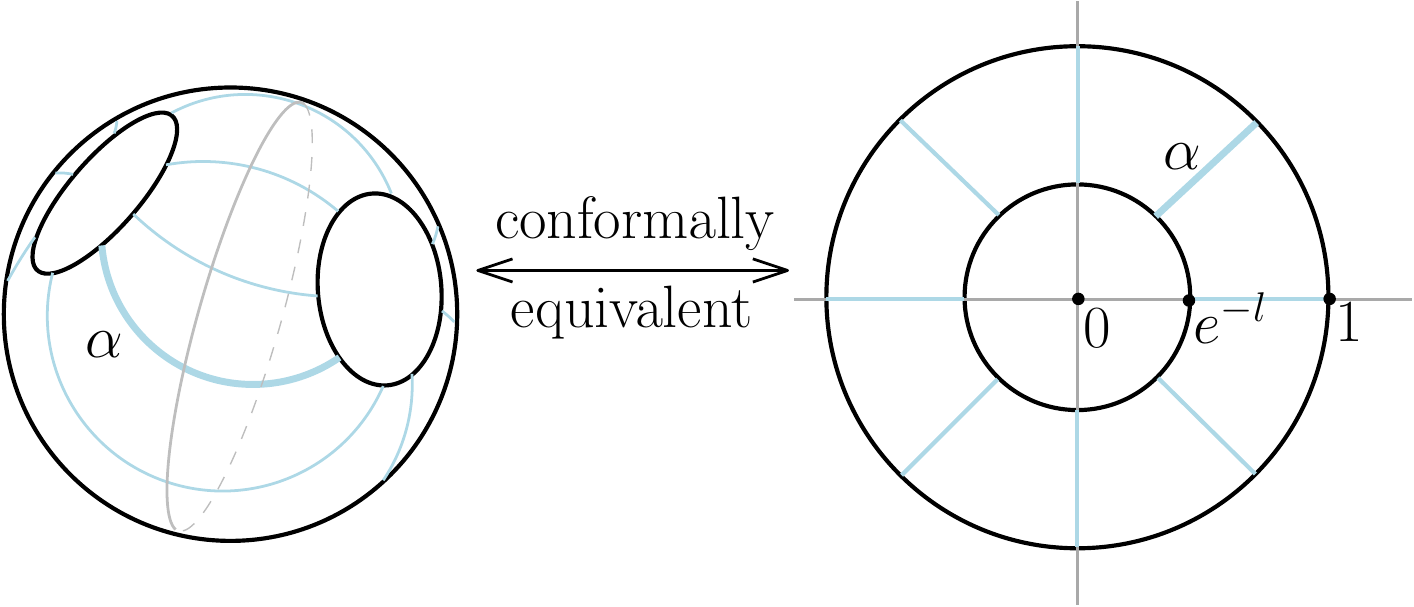}
\caption{Conformally equivalent M\"obius annuli.}
\label{MobiusAnnulus}
\end{figure}

\begin{figure}[h]
\centering
\includegraphics[width=11cm]{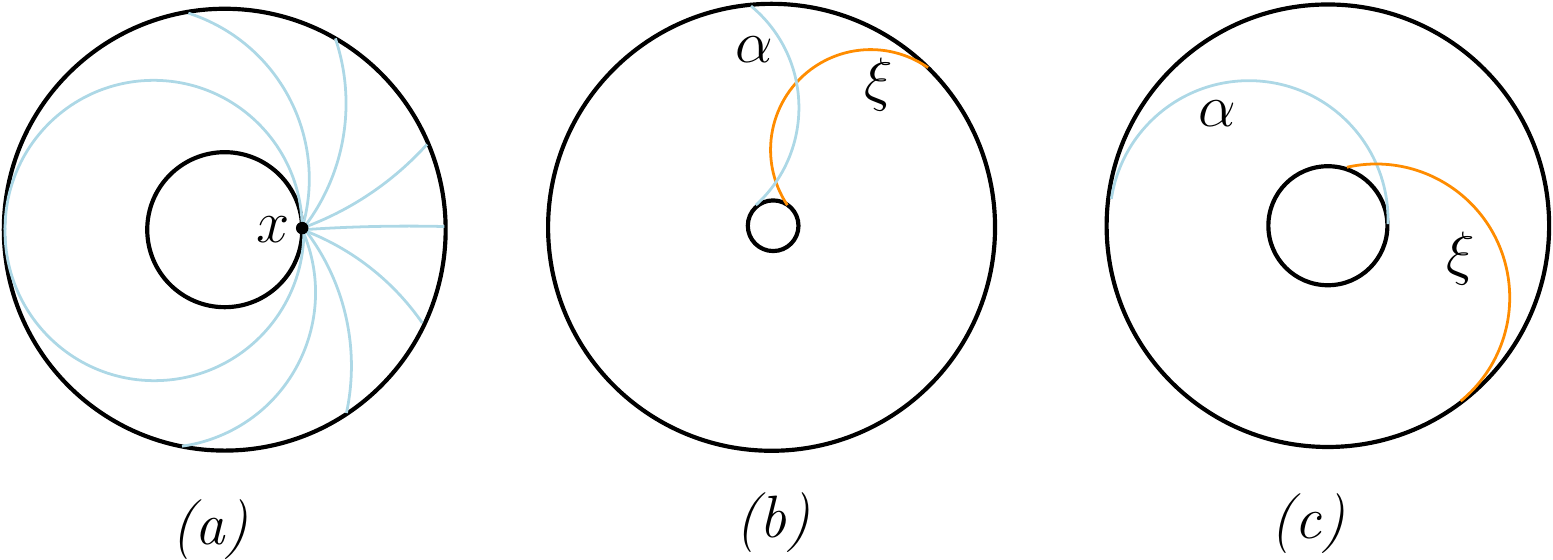}
\caption{\emph{(a)} circle arcs in $\mathcal{A}$ connecting $x$ to every point in the other boundary component, \emph{(b)} circle arcs $\alpha,\xi$ with $i(\alpha,\xi)=0$ and \emph{(c)} with $i(\alpha,\xi)=1$.}
\label{RemarkJustifications}
\end{figure}

\begin{definition}
Let $\mathcal{A}$ be a M\"obius annulus and $x$ and $y$ be two points on the boundary components $\partial_1\mathcal{A}$ and $\partial_2\mathcal{A}$ respectively. We say $x, y$ is a \emph{singular pair} on the boundary of $\mathcal{A}$ if there is a non-contractible circle $C$ subset of the closure $cl({\mathcal{A}})$ so that $C$ contains both $x$ and $y$.
\end{definition}
Note that each point on $\partial\mathcal{A}$ belongs to a unique singular pair.

\begin{lemma}
\label{remLem1}
(See Figure \ref{RemarkJustifications}.) We call a connected segment of a circle a \emph{circle arc}. Let $\partial_1 \mathcal{A}$ and $\partial_2 \mathcal{A}$ be the boundary components of a M\"obius annulus $\mathcal{A}$. Then:
\begin{enumerate}
\item Every point from $\partial_1 \mathcal{A}$ can be connected to any point on $\partial_2 \mathcal{A}$ by a circle arc on $\mathcal{A}$;
\item if $x,y$ are a singular pair on $\partial\mathcal{A}$, then there are two circle arcs not in the same homopoty class in $cl(\mathcal{A})$ connecting $x$ and $y$.
\item if $x,y$ are on two different components of $\partial\mathcal{A}$ but are not a singular pair, then all circle arcs connecting $x$ and $y$ belong to the same homotopy class (rel. endpoints) in $cl(\mathcal{A})$.
\end{enumerate}
\end{lemma}

\begin{proof}

As suggested in figure \ref{RemarkJustifications} we can find a family of circle arcs in $cl(\mathcal{A})$ connecting a point on $\partial_1 \mathcal{A}$ to any other point on $\partial_2 \mathcal{A}$. Without loss of generality, we assume the annulus $\mathcal{A}$ is given by $\mathcal{A}=\{z\in \mathbb{C}\ |\  e^{-l} < |z| < 1 \}$. Let $\alpha$ be an oriented circle arcs on $\mathcal{A}$ connecting the two boundary components. Using polar decomposition, let $r:\mathbb{C}\rightarrow (0,\infty)$ be the function giving the distance between 0 and a point on $\mathbb{C}$, and let $\theta$ be the multivalued function representing the argument of a complex number. We have $d\theta$ is a closed 1-form on $\mathbb{C}-\{0\}$. For any curve $\gamma$ not passing thru 0, the winding number of $\gamma$ is
$$w(\gamma) = \frac{1}{2\pi}\int_\gamma d\theta.$$

This allows us to define the winding number of a curve on a M\"obius annulus $\mathcal{A}$ because Inv$^*(\mathcal{A}) \cong SO(2)$ in the planar picture are Euclidean rotations around 0. This is well-defined up to an orientation on the annulus. Since $\alpha:[0,1]\rightarrow cl({\mathcal{A}})$ is a circle arc connecting two boundary components, $\alpha$ is part of a circle $C_\alpha\subset \mathbb{C}$ that's not centered at 0. So there is a point $z_{min}$ on $C_\alpha$ that's closest to $0$, and a point $z_{max}$ on $C_\alpha$ that is furthest from $0$. It's easy to see that $z_{min}, z_{max}$ and $0$ lie on a straight line in the plane, and $z_{min}, z_{max}$ uniquely determine the circle $C_\alpha$. So $\alpha$ is a subset of a half circle connecting $z_{min}$ and $z_{max}$. We use the notation $\frac{1}{2}C_\alpha$ to refer to this half circle. We have the winding number $w(\frac{1}{2}C_\alpha) = \pm \frac{1}{2}$ or $0$, and thus any subsegment $\alpha \subset \frac{1}{2}C_\alpha$ has winding number $|w(\alpha)|\le 1/2$. We have equality if and only if $\alpha=\frac{1}{2}C_\alpha$ and the endpoints of $\alpha$ are $z_{min}, z_{max}$.

Now let $\alpha, \beta$ be two circle arcs with the same endpoints on $\partial \mathcal{A}$. Note that $\alpha\beta^{-1}$ is a closed loop with integer winding number, and  $|w(\alpha\beta^{-1})| = |w(\alpha) - w(\beta)|\le \frac{1}{2}+\frac{1}{2} =1$. We have $\alpha$ and $\beta$ are not homotopic in $cl(\mathcal{A})$ only when $|w(\alpha\beta^{-1})| =1$. This happens only when $w(\alpha)=w(\beta)=\pm 1/2$ so $\alpha$ and $\beta$ are two halves of the same circle $C\subset \mathcal{A}$, which means the end-points are a singular pair. Otherwise, if the endpoints of $\alpha,\beta$ are are not singular pair then $w(\alpha\beta^{-1}) = 0$ and $\alpha,\beta$ are homotopic rel. end points.
\end{proof}

\begin{definition}
Given a M\"obius annulus $\mathcal{A}$, we define a radial orientation on it to be an equivalence class of bijective conformal maps $\mathcal{A}\rightarrow \{ z\in\mathbb{C}\ |\ e^{-\textnormal{mod}(A)}<|z|<1 \}$ where two maps are equivalent if they are related by a composition with an $SO(2)$ rotation.
\end{definition}
 The difference between two such orientation is an inversion on $\mathbb{C}$ that exchanges the two boundary components of $\{ z\in\mathbb{C}\ |\ e^{-\textnormal{mod}(A)}<|z|<1 \}$. Basically, a radial orientation of $\mathcal{A}$ defines which boundary component of $\mathcal{A}$ is the inner and which is the outer one.

\begin{definition}
\label{defNaturalFibration}
Given a M\"obius annulus $\mathcal{A}$ which is conformally equivalent to  $\{ z\in\mathbb{C}\ |\ e^{-l}<|z|<1 \}$, the \emph{natural $S^1$-fibration} of $\mathcal{A}$ is a continuous map $f:\mathcal{A}\rightarrow [e^{-l},1]$ such that the fibers are concentric circles :$f^{-1}(x) = \{ z\in\mathbb{C}:|z|=x \}$.
\end{definition}
 
\begin{definition}
A \emph{marking} on a radially oriented M\"obius annulus $\mathcal{A}$ is a homotopy class $[\alpha]$ (rel. end points) of curves $\alpha:[0,1]\rightarrow cl({\mathcal{A}})$ such $\alpha(0)$ is in the inner boundary component and $\alpha(1)$ is in the outer one, and that the winding number $|w(\alpha)|\le 1/2$. The pair $(\mathcal{A},[\alpha])$ is then called a marked annulus.
\end{definition}
Note that the definition of winding number is in the proof of lemma \ref{remLem1}. The lemma implies that each marking has a circle arc representative (not unique). Moreover, a marking $[\alpha]$ on a radially oriented M\"obius annulus is completely determined by one end-point of $\alpha$ and the winding number $w(\alpha)$, a real number in $[-1/2, 1/2]$.

\section{Flat conformal structures and circle bundles}
\label{SecFunDomain}

We will first define the \emph{flat conformal structure} on a manifold. Following Thurston's more general notion of $(G,X)$ structures as in \cite{ThurstonNotes} , we let $G=\Mob(S^3)$ and $X=S^3$. Let $M$ be a 3-manifold. Then a $(\Mob(S^3),S^3)$ structure on $M$ is an atlas with charts from open sets of $M$ to $S^3$, and transition maps are restrictions of actions by $\Mob(S^3)$ elements. We will call this a flat conformal structure on $M$.

This more rigid type of manifold structure allows us to extend charts along curves and define a developing map on the universal cover $\textnormal{dev}: \widetilde{M} \rightarrow S^3$ which is a local diffeomorphism. We also get a holonomy representation $\rho:\pi_1(M)\rightarrow \Mob(S^3)$ which is equivariant with respect to $\textnormal{dev}$. That is, for $A\in\pi_1(M)$ acting by deck transformation on $\widetilde{M}$, $\textnormal{dev}(A.p) = \rho(A).\textnormal{dev}(p)$. From an equivariant pair $(\textnormal{dev},\rho)$ one can construct an atlas as in the definition. So a flat conformal structure can be seen as a pair $(\textnormal{dev},\rho)$ satisfying the above conditions.

\subsection{Circle bundles, combinatorial construction}
\label{SecCircleBundles}

The 3-manifolds in which we are interested are total space of oriented circle bundles with structure group $\textnormal{Homeo}^+(S^1)$. Let $E$ be the total space of such circle bundle. Homeomorphisms of the circle can be extended to the unit disc, so we have an associated disc bundle. The Euler number of $E$ can be viewed as the self-intersection number of a section of the associated disc bundle. This is the point of view taken in \cite{GLT} as they estimate the Euler number of conformally flat circle bundles coming from tesselations by regular polyhedra.  

There is another equivalent formulation of the Euler number coming from the fundamental group of the total space $E$ which we will present below. Let $\Sigma_g$ be the topological closed surface of genus $g$. Let $D_{4g}$ be a (2-dimensional) polygon of $4g$ sides and suppose the surface is identified with the polygon $D_{4g}$ with its sides glued under the standard identification pattern $\sim_A$. That is, when we list the sides of polygon $D_{4g}$ in clockwise order we get the word $W_A(A_1,...,A_{2g}) = A_1 A_2 A_1^{-1} A_2^{-1} ... A_{2g-1} A_{2g} A_{2g-1}^{-1} A_{2g}^{-1}$. So we have $\Sigma_g \cong D_{4g}/\sim_A$ and the standard presentation of the fundamental group:

$$\pi_1(\Sigma_g,p) =\left\langle A_1,..., A_{2g}\ |\ W_A(A_1,...,A_{2g}) = 1  \right\rangle$$
(with a slight abuse of notation $A_i$.) We choose nice representative loops $a_i:[0,1]\rightarrow \Sigma_g$ of $A_i$ so that $a_i$ does not intersect $a_j$ except at the base point $p$. We have $\Sigma_g - \cup_i a_i$ can be identified with the interior $int(D_{4g})$.

Let $E\stackrel{f}{\rightarrow}\Sigma_g$ be an orientable circle bundle. Pick any lift $\hat{a}_i$ of $a_i$ so that if $\hat{a}_i$ are closed loops. That is, we have $a_i = p\circ\hat{a}_i$. Let $\hat{A}_i$ be the homotopy class of $\hat{a}_i$. Let $C$ be the homotopy class of the fiber loop over $p$ and we pick a base point $\hat{p}\in E$ over $p\in\Sigma_g$. Since $E$ is an orientable bundle, the bundle restricted to a loop $a_i$ is a torus embedded in $E$. So we have relations $[\hat{A_i},C]=1$ coming from the embbeded tori.  There are $2g$ tori coming from the $2g$ generating loops on the surface. We cut along these tori and we are left with a circle bundle over $int(D_{4g})$ which must be trivial and thus identified with $int(D_{4g})\times S^1$, this space has a natural closure: $D_{4g}\times S^1$.  Note that $D_{4g}\times S^1$ has $4g$ boundary pieces, each a topological annulus, which are paired up and identified by a collection of maps $\sim_{\bf A}$ to recover the circle bundle $E$, that is $E \cong D_{4g}\times S^1/\sim_{\bf A}$. By van Kampen theorem we get:
$$\pi_1(E) =\left\langle \hat{A_1},...,\hat{A_2g}, C\ |\ [\hat{A_i},C]=1, W_A(\hat{A_1},...,\hat{A}_{2g}) = C^k  \right\rangle$$
for some $k\in\mathbb{Z}$.
\begin{remark} In the present article we define this number $k$ to be the Euler number $e(E)$ of circle bundle $E\stackrel{f}{\rightarrow} \Sigma_g$.
\end{remark}

 Note that $k$ does not depend on the choice of generators $\hat{a_i}$. If we choose a different lift $\hat{a}'_i$ over $a_i$, then $\hat{a}'_i$ is also a loop in the torus over $a_i$, and $\hat{A}'_i \simeq \hat{A_i} C^{a_i}$ for some integer $a_i$. We have $W(\hat{A_1},...,\hat{A}_{2g})$ contains $\hat{A_i}$ and $\hat{A_i}^{-1}$ exactly once, and $C$ commutes with everything in $\pi_1(E)$, so replacing $\hat{A_i}$ with $\hat{A'_i}$ does not change the relation and the invariant $e(E)$. 

Now we will show it is possible to compute $e(E)$ from different (non-standard) presentations of $\pi_a(\Sigma_g)$ and $\pi_1(E)$. Let $D_n$ be a polygon of $n$ (even) sides and $\sim_B$ is an identification of its sides so that the closed surface $\Sigma_g$ is identified with $D_n /\sim_B$. As a side note, we have $E\cong D_n\times S^1/\sim_{\bf B}$ by a similar construction as before. We name the sides $B_1,...,B_{n/2}, B_1^{-1},...,B_{n/2}^{-1}$ and let $W_B(B_1,...,B_{n/2})$ be the word obtained by listing the edges of $P_n$ in clockwise order. Let $b_i$ be nice paths on the surface $\Sigma_g$ that represent $B_i$. Note that $b_i$ are not necessarily closed loops, but paths connecting between a collection of base points $\{p_1,...,p_m\}$. We choose lifts $\{\hat{p}_1,...,\hat{p}_m\}$ over these points, and we choose a lifts $\hat{b}_i$ over $b_i$ so that they connect between points in $\{\hat{p}_1,...,\hat{p}_m\}$. We have the homotopy class $[W_B(\hat{b}_1,...,\hat{b}_{n/2})] = C^{l}$ for some integer $l$. Next we will show that $k=l$.

 The concatenations of $b_i$'s generate all homotopy classes of loops on $\Sigma_g$. So we let $\alpha_i(b_1,...,b_{n/2})$ be the concatenated loop that is homotopic to $a_i$. We assume that the base point $p$ of $a_i$ is the same as the starting base point of $b_1$. We have $\alpha_i(\hat{b}_1,...,\hat{b}_{n/2})$ is a lift of $\alpha_i(b_1,...,b_{n/2})$ and $\alpha_i(b_1,...,b_{n/2})$ is homotopic to $a_i$. So by homotopy lifting property, $\alpha_i(\hat{b}_1,...,\hat{b}_{n/2})$ is homotopic to a lift $\hat{a}_i$ over $a_i$. Therefore $W_A[\alpha_1(\hat{b}_1,...,\hat{b}_{n/2}),...,\alpha_{4g}(\hat{b}_1,...,\hat{b}_{n/2}) ]$ is homotopic to $W_A(\hat{a}_1,...,\hat{a}_{4g})\simeq C^k$.
 
 Without loss of generality, we can assume $W_A[\alpha_1({b}_1,...,{b}_{n/2}),...,\alpha_{4g}({b}_1,...,{b}_{n/2}) ]$ and $W_B(b_1,...,b_{n/2})$ both starts with $b_1$. Let $x_0\in\Sigma_g - (\cup_i a_i) - (\cup_j b_j)$ so that $\Sigma_g - \{x_0\}$ deformation retracts to either the $a_i$ skeleton or the $b_i$ skeleton. So we have the following loops $W_A[\alpha_1({b}_1,...,{b}_{n/2}),...,\alpha_{4g}({b}_1,...,{b}_{n/2}) ]$ and $W_A(a_1,...,a_{4g})$ and $W_B(b_1,...,b_{n/2})$ are all homotopic (to a simple loop around $x_0$) in  $\Sigma_g - \{x_0\}$. Through the deformation retract from $\Sigma_g -\{x_0 \}$ to the $b_i$ skeleton, we have $W_A[\alpha_1({b}_1,...,{b}_{n/2}),...,\alpha_{4g}({b}_1,...,{b}_{n/2}) ]$  and $W_B(b_1,...,b_{n/2})$ are homotopic in the $b_i$ skeleton $S^{(1)}_B$ whose fundamental group $\pi_1(S^{(1)}_B)$  is a free group. Therefore $W_B(b_1,...,b_{n/2})$  must be obtained from $W_A[\alpha_1({b}_1,...,{b}_{n/2}),...,\alpha_{4g}({b}_1,...,{b}_{n/2})]$ by removing overlaps such as $b_ib_i^{-1}$, because otherwise we would obtain a non-trivial relator for $\pi_1(S^{(1)}_B)$. This implies the homotopy
 
 $$ W_A[\alpha_1(\hat{b}_1,...,\hat{b}_{n/2}),...,\alpha_{4g}(\hat{b}_1,...,\hat{b}_{n/2})] \simeq W_B(\hat{b}_1,...,\hat{b}_{n/2})$$
 and thus we have
 
 \begin{remark}
 $$
 C^k = [W_A[\alpha_1(\hat{b}_1,...,\hat{b}_{n/2}),...,\alpha_{4g}(\hat{b}_1,...,\hat{b}_{n/2})]] = [W_B(\hat{b}_1,...,\hat{b}_{n/2})] = C^l.
 $$
 so $l=k=e(E)$.
\end{remark}

\subsection{Structures with fundamental domains}
\label{SecFundomain}

We will only consider conformal structures on $E$ with holonomy representation $\rho: \pi_1(E)\stackrel{}{\rightarrow}\Mob(S^3)$ that factors into 
$$\begin{array}{ccc}
\pi_1(E) & \rightarrow &\Mob(S^3) \\
\downarrow & \nearrow \\
\pi_1(\Sigma_g)
\end{array}$$
In other words, the fiber generator $C\in\pi_1(E)$ is mapped to $\rho(C) = 1$. Let $\Gamma = \rho(\pi_1(E)) \subset\Mob(S^3)$ be the image of $\rho$, so $\Gamma$ is isomorphic to a fundamental group of a surface, $\Gamma$ is said to be a surface group in $\Mob(S^3)$.

\begin{figure}[h]
\centering
\includegraphics[width=10cm]{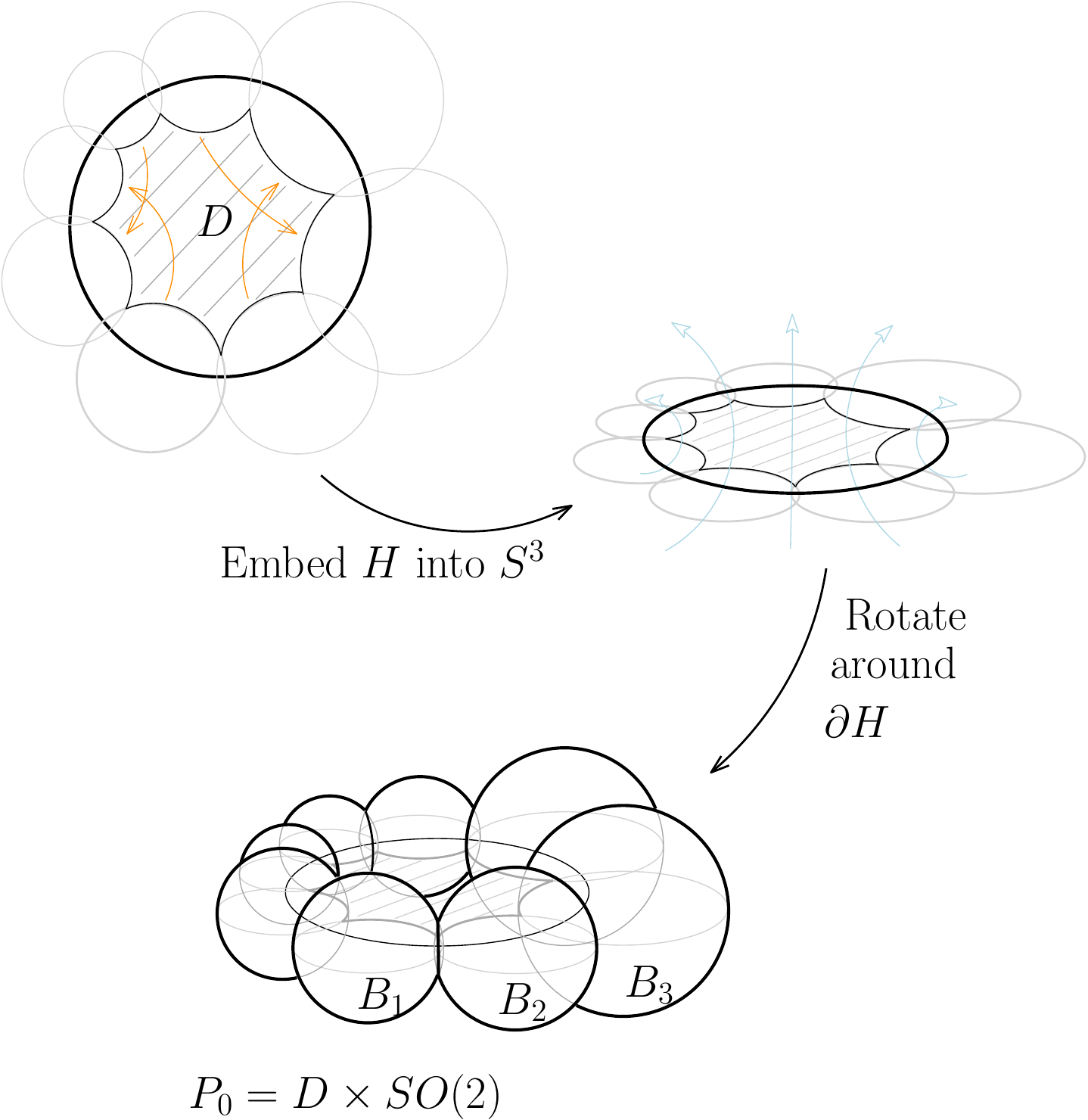}
\caption{Example 1, Fuchsian example illustration. In this picture, infinity is inside $P_0$.}
\label{fig:Example1}
\end{figure}

\emph{Example 1.} 
An easy example of a conformally flat manifold can be described as follows. Take a hyperbolic surface that has fundamental domain $D$ with a \emph{standard identification pattern}; embed the unit disk model of $\mathbb{H}^2$ into $S^3$ as a half-sphere $H$; then rotate $H$ by the singular-elliptic group $\textnormal{Fix}(\partial H)$. The domain $D$ under $\textnormal{Fix}(\partial H)$-action will sweep out a polyhedron in $S^3$ which we denote by $P_0$, and we have $P_0 = S^3 - \cup_{i=1}^n B_i$ where $B_i$ are open balls. So the faces of $P_0$ are all \emph{aligned}, that is, the spheres $\partial B_1, \partial B_2,... , \partial B_n$ are all orthogonal to the same circle $\partial H$. (Note that by open ball we mean the open connected region in $S^3$ that's bounded by a 2-sphere.)

This example is basically obtained from a totally geodesic embedding $\mathbb{H}^2\hookrightarrow\mathbb{H}^4$ along with an embedding of Fuchsian group $\Gamma_0\hookrightarrow PSL(2,\mathbb{R})\hookrightarrow \Mob^+(S^3)$. $P_0$ is a fundamental domain for the action of $\Gamma_0$ on $S^3 = \partial_\infty\mathbb{H}^4$. Moreover, $(S^3 - \partial H) /\Gamma_0$ is conformally flat trivial bundle over a surface.

\bigskip

\emph{Example 2.}
The first examples of non-trivial circle bundles with flat conformal structures were constructed by Gromov-Lawson-Thurston \cite{GLT}, Kapovich \cite{Kapovich89Flat} and Kuiper \cite{Kuiper88}. All these examples are constructed by fundamental domains and/or tesselations of $\mathbb{H}^4$.


Next, some definitions.

\begin{definition}
A polyhedron in $S^1$ is either a circle arc bounded by 2 points or the whole circle. A \emph{polyhedron} $P\subset S^m$ is a closed region with non-empty interior $int(P)$ such that:
\begin{itemize}
\item $cl(int(P))=P$,
\item the boundary $\partial P$ is a union of polyhedra in $S^{m-1}$ such that the intersection between two of them are either empty or polyhedra in $S^{m-2}$.
\end{itemize}
The codimension 1 pieces of the boundary are called \emph{faces} and the codimension 2 pieces are called \emph{edges}.
\end{definition}

\begin{definition}
Let $P\subset S^m$ be a polyhedron of dimension $m$. We use $\{P^{m-k}\}$ to denote the set of codimension $k$ polyhedra on the boundary of $P$. We can also call it the combinatorial $(m-k)$-skeleton of $\partial P$.
\end{definition}

\begin{definition}
A polyhedron $P\subset S^m$ is \emph{convex} if $P = S^m - \cup B_i$ where each $B_i$ is a ball with boundary an $(m-1)$-sphere. That is $P$ is the intersection of a collection of half spaces.
\end{definition}

\begin{figure}[h]
\centering
\includegraphics[width=5cm]{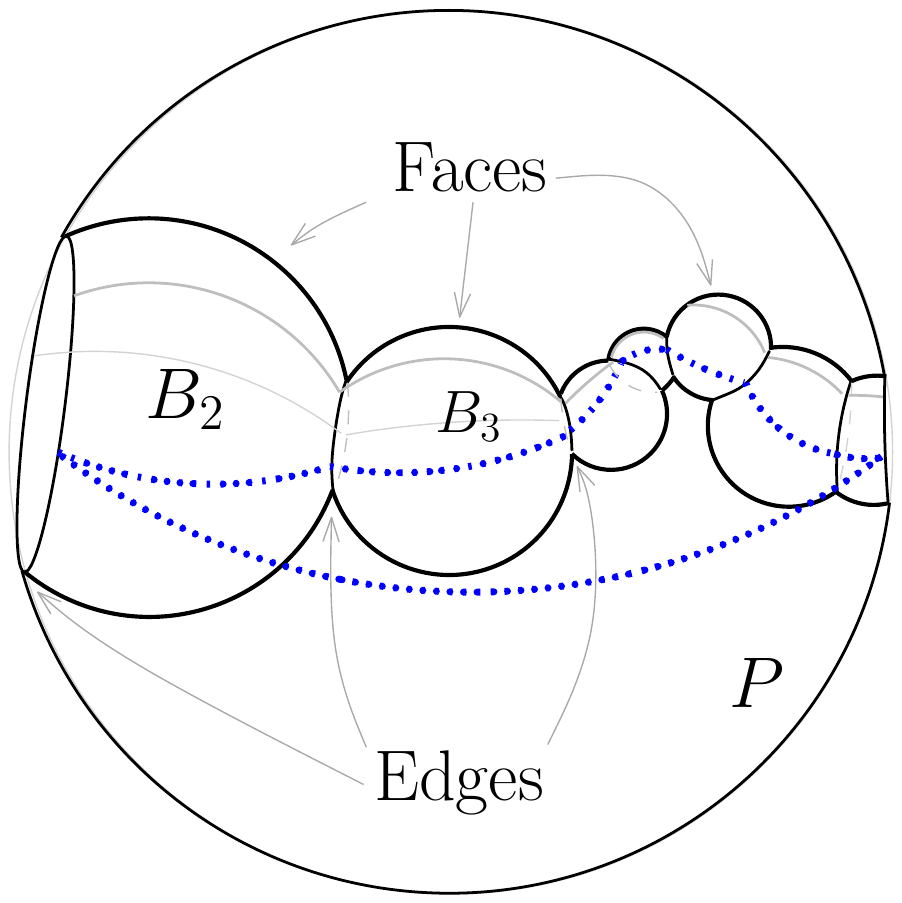}
\caption{A solid torus polyhedron $P$ with $n=8$ faces. In this picture, infinity is in $S^3 - P$.}
\label{fig:mobiuspoly}
\end{figure}

\noindent Following \cite{AG11} we define the following:
\begin{definition}
A \emph{cornerless} polyhedron is one whose boundary contains only codimension 1 and codimension 2 pieces.
\end{definition}

\noindent

Let $Q$ be a (2-dimensional) polygon with $n$ edges and hence $n$ vertices. Let $\sim_Q$ be an identification of the edges of $Q$ such that $Q/\sim_Q$ is topologically a closed orientable surface. Let $\Sigma_Q = Q/\sim_Q$ be this topological surface. There is an equivalence relations on the set of edges $\{Q^1\}$ and on the set of vertices $\{Q^0\}$ of $Q$ induced by the above identification; abusing notation we will call this equivalence relation $\sim_Q$ as well.

Let $P$ be a (cornerless) solid torus polyhedron with $n$ faces so that the faces are M\"obius annuli. Let $\sim_P$ be an equivalence relation on the set of faces $\{P^2\}$ and the set of edges $\{P^1\}$ of $P$. We say that $(P,\sim_P)$ and  $(Q,\sim_Q)$ are \emph{combinatorially equivalent} if and only if there's a bijective map  $f: \{P^2\}\rightarrow\{Q^1\}$, $f: \{P^1\}\rightarrow\{Q^0\}$ such that $P^k_i \sim_P P^k_j$ exactly when $f(P^k_i) \sim_Q f(P^k_j)$ for $k=1,2$ and $i,j\in \mathbb{Z}/n\mathbb{Z}$.

Note that even though the topological surface $Q/\sim_Q$ is well defined, $P/\sim_P$ is not well-defined if we only have the combinatorial information: equivalence relations on the set of faces $\{P^2\}$ and on the set of edges $\{P^1\}$. We would need actual homeomorphisms between paired faces to define $P/\sim_P$ as a manifold using quotient topology.

\begin{remark}
The natural $S^1$-fibration of $\partial P$ is a continuous map $\partial P\rightarrow \partial Q$ that extends the natural $S^1$-fibration on each face of $P$. See definition \ref{defNaturalFibration}.
\end{remark}


\begin{definition}
The \emph{domain of discontinuity} of $\Gamma\subset SO(4,1)$, denoted by $\Omega_\Gamma$, is defined to be the largest open subset of $S^3$ on which $\Gamma$ acts properly discontinuously. The \emph{limit set} $\Lambda_\Gamma$ of $\Gamma$ is defined to be $S^3 - \Omega_\Gamma$.
\end{definition}

\begin{theorem}
\label{theoremGammaEuler}
 Let $P\subset S^3$ be a corneless solid torus polyhedron with $n$ faces and suppose $P$ is a fundamental domain for a surface group $\Gamma\subset\Mob^+(S^3)$ acting on $S^3 - \Lambda_\Gamma$. Additionally suppose that the face-pairing transformations in $\Gamma$ define an equivalence relation $\sim_P$ on the set of faces and the set of edges of $P$ combinatorially equivalent to $\sim_Q$ on the set of edges and the set of vertices of a polygon $Q$ that results in an orientable closed surface. Conclusion: Then $E = (S^3 - \Lambda_\Gamma) / \Gamma$ is a conformally flat circle bundle over a closed surface, and there is a loop $\gamma\subset \partial P$ composed of $n$ circular arcs such that $[\gamma] = {e(E)}\in \pi_1(P)\cong\mathbb{Z}$ (with appropriate orientation for the generator of $\pi_1(P)$).
\end{theorem}
\begin{proof}
Following \cite{GLT}(section 5 and 7) we have $\Lambda_\Gamma$ is a topological circle in $S^3$, and $S^3-\Lambda_\Gamma$ is homeomorphic to a solid torus.

The fundamental domain condition implies that we have a tesselation of $S^3 - \Lambda_\Gamma$ by action of $\Gamma$ on $P$. The side-pairing M\"obius tranformations are unique, and they realize $\sim_P$, thus manifold $E=P/\sim_P$ can be defined. Moreover $E=P/\sim_P = (S^3 - \Lambda_\Gamma) / \Gamma$ a conformally flat manifold because $P$ is a fundamental domain for $\Gamma$.

We can define a fibration $S^1\rightarrow P\rightarrow Q$ extending the natural $S^1$-fibration $\partial P \rightarrow \partial Q$. Note that the side-pairing M\"obius tranformations preserve the natural $S^1$-fibration of $\partial P$. In addition, $(P,\sim_P)$ is combinatorially equivalent to $(Q,\sim_Q)$, so we get a fibration $S^1\rightarrow (P/\sim_P) \rightarrow (Q/\sim_Q)$. Therefore $E$ is a circle bundle over the closed surface $\Sigma_Q = Q/\sim_Q$.

Now we will describe an algorithm to construct the loop $\gamma$. Let  $\mathcal{A}_1...\mathcal{A}_n$ be the faces of $P$ and $E_1,...,E_n$ be the edges so that for $i\in\mathbb{Z}/n\mathbb{Z}$ we have $\mathcal{A}_i, \mathcal{A}_{i+1}$ are adjacent and share an edge: $E_{i+1}$, which means $\mathcal{A}_i$ contains edges $E_i, E_{i+1}$ for $i\in\mathbb{Z}/n\mathbb{Z}$. (The faces and the vertices of $P$ are cyclically ordered.) Suppose edges $E_{i_1},...,E_{i_m}$ are identified under face pairing M\"obius transformations. We pick a point $p_{i_1}\in E_{i_1}$. The identification maps give us $p_{i_k} \in E_{i_k}$ for $k=2,...,m$. Note that these are lifts of the same point in $E$. We can do the above for every other equivalence class of edges. Now we have a point $p_i$ on each edge $E_i$ for $i=1,...,n$. Face $\mathcal{A}_i$ contains edges $E_i, E_{i+1}$ for $i\in\mathbb{Z}/n\mathbb{Z}$, so we connect $p_i$ to $p_{i+1}$ by a circle arc $\gamma_i: [0,1]\rightarrow\mathcal{A}_i$. This is possible by lemma \ref{remLem1}. Suppose $A_i$ is identified with $A_j$ by the M\"obius transformation $B_{i,j}$, then we have $B_{i,j}(p_i) =p_{j+1}$ and $B_{i,j}(p_{i+1}) = p_{j}$. Then let $\gamma_j = B_{i,j}(\gamma_i^{-1})$ which is a circle arc connecting $p_{j}$ to $p_{j+1}$.  We can repeat this process until there is a circle arc on each face connecting $p_1,...,p_n$, and therefore the concatenation $\gamma = \gamma_1\gamma_2...\gamma_n$ is a closed loop.

Consider the quotient map $p: P\rightarrow (P/\sim_P) = E$, and the induced $p^*:\pi_1(P)\rightarrow\pi_1(Q)$, and also the quotient map $q: Q\rightarrow Q/\sim_Q = \Sigma_Q$ and the induced $q^*: \pi_1(Q)\rightarrow\pi_1(\Sigma_Q)$. Let $c$ be a loop in $P$ so that $[c]$ generates $\pi_1(P)$. By construction, $p^*([c])\in\pi_1(E)$ is the generator corresponding to the fiber of $E$. So we let $C = p^*([c])$ the fiber generator.

Let $a_1,...,a_n$ be the edges of the polygon $Q$ corresponding to $\mathcal{A}_1,...,\mathcal{A}_j$ under the equivalence. If $\mathcal{A}_i, \mathcal{A}_j$ are identified faces then $p(\gamma_i)=p(\gamma_j^{-1})$ in $E$, and also $p(a_i)=p(a_j^{-1})$ in $\Sigma_Q$. We have $p(\gamma_i)$ is a lift of $q(a_i)$ for $i=1,...,n$. So following the discussion in section \ref{SecCircleBundles} we have $p^*([\gamma]) = [p(\gamma)]=[p(\gamma_1)...p(\gamma_n)]$ by definition. But $p(\gamma_1)...p(\gamma_n)$ is a lift of a homotopically trivial loop on the surface, so $[p(\gamma_1)...p(\gamma_n)] = C^k = p^*([c])^k = p^*([c]^k) \in\pi_1(E)$ for some integer $k$. Thus $p^*([\gamma])= p^*([c]^k)$ which implies $[\gamma] = [c]^k \in\pi_1(P)$ since $p^*$ is injective. Moreover $e(E) = k$ as discussed in section \ref{SecCircleBundles}, the theorem follows.
\end{proof}
As a corollary, with the same hypothesis as above implies $|e(E)|<\frac{3}{2}n^2$. The proof will be in section 4.


\label{exampleNew}
\emph{Example 3.}
First let us describe a new non-trivial example with Euler number computation. A new feature is that there will be side-pairing M\"obius transformations that are loxodromic without rotation.

Let $S_1, S_2, ..., S_{25}$ be open Euclidean spheres in $\mathbb{R}^3 = S^3 - \{\infty \}$ centered at 
$\left(\begin{array}{c}0 \\ 0 \\ 0 \end{array}\right)$, 
$\left(\begin{array}{c}1 \\ 0 \\ 0 \end{array}\right)$, 
$\left(\begin{array}{c}1 \\ 0 \\ 1 \end{array}\right)$, 
$\left(\begin{array}{c}1 \\ 0 \\ 2 \end{array}\right)$, 
$\left(\begin{array}{c}1 \\ 1 \\ 2 \end{array}\right)$, 
$\left(\begin{array}{c}1 \\ 2 \\ 2 \end{array}\right)$, 
$\left(\begin{array}{c}1 \\ 2 \\ 1 \end{array}\right)$, 
$\left(\begin{array}{c}1 \\ 2 \\ 0 \end{array}\right)$, 
$\left(\begin{array}{c}2 \\ 2 \\ 0 \end{array}\right)$, 
$\left(\begin{array}{c}3 \\ 2 \\ 0 \end{array}\right)$, 
$\left(\begin{array}{c}3 \\ 1 \\ 0 \end{array}\right)$, 
$\left(\begin{array}{c}3 \\ 0 \\ 0 \end{array}\right)$, 
$\left(\begin{array}{c}4 \\ 0 \\ 0 \end{array}\right)$, 
$\left(\begin{array}{c}5 \\ 0 \\ 0 \end{array}\right)$, 
$\left(\begin{array}{c}5 \\ 0 \\ 1 \end{array}\right)$, 
$\left(\begin{array}{c}5 \\ 0 \\ 2 \end{array}\right)$, 
$\left(\begin{array}{c}5 \\ 1 \\ 2 \end{array}\right)$, 
$\left(\begin{array}{c}5 \\ 2 \\ 2 \end{array}\right)$, 
$\left(\begin{array}{c}5 \\ 2 \\ 1 \end{array}\right)$, 
$\left(\begin{array}{c}5 \\ 2 \\ 0 \end{array}\right)$, 
$\left(\begin{array}{c}6 \\ 2 \\ 0 \end{array}\right)$, 
$\left(\begin{array}{c}7 \\ 2 \\ 0 \end{array}\right)$, 
$\left(\begin{array}{c}7 \\ 1 \\ 0 \end{array}\right)$, 
$\left(\begin{array}{c}7 \\ 0 \\ 0 \end{array}\right)$,
$\left(\begin{array}{c}8 \\ 0 \\ 0 \end{array}\right)$,
 respectively.
We choose $r_1$ to be the radius for $S_i$ for odd $i$ between 1 and 25, and $r_2$ to be the radius for $B_i$ for even $i$ between 1 and 25. It is possible to choose $r_1, r_2$ so that only adjacent spheres intersect. For $i=1,...,24$ we let $E_{i+1} = S_i \cap S_{i+1}$ which are all Euclidean circles. Notice that by construction, the radii of $E_i$ are the same and its value depends on $r_1, r_2$. Let $r$ be the radius of $E_i$, technically $r$ is a function $r(r_1,r_2)$. Let $E_1$ be the image of $E_2$ under the reflection 
$\left(\begin{array}{c}x \\ y \\ z \end{array}\right) \mapsto \left(\begin{array}{c}-x \\ y \\ z \end{array}\right)$.
 Let $E_{26}$ be the image of $E_{25}$ under the reflection 
$\left(\begin{array}{c}x \\ y \\ z \end{array}\right) \mapsto \left(\begin{array}{c}16-x \\ y \\ z \end{array}\right)$.
Let $S_{0} = S_{28}$ be the sphere centered at $\left(\begin{array}{c}-d \\ 0 \\ 0 \end{array}\right)$ that contains $E_1$, and let $S_{26}$ be the ball centered at $\left(\begin{array}{c}8+d \\ 0 \\ 0 \end{array}\right)$ that contains $E_{26}$. Note that $S_0$ and $S_{26}$ are symmetric from the mid point $\left(\begin{array}{c}4 \\ 0 \\ 0 \end{array}\right)$. Let ${S}_{27}$ be a sphere centered at $\left(\begin{array}{c}4 \\ 0 \\ 0 \end{array}\right)$ with radius $R$ so that $S_{27}$ only intersect $S_{26}$ and $S_0$, and let $E_{27} = S_{26}\cap S_{27}$, $E_0=E_{28} = S_{27}\cap S_0$. For $i\in\mathbb{Z}/28\mathbb{Z}$ let $\mathcal{A}_i$ be the M\"obius annulus contained in $S_i$, bounded by $E_i, E_{i+1}$. If $d$ is large enough, we can choose radius $R$ such that $\textnormal{mod}(\mathcal{A}_{25}) = \textnormal{mod}(\mathcal{A}_{27})$, so $R$ is completely determined by $r_1,r_2,d$. Let $P=P(r_1,r_2,d)$ be the conformal polyhedron bounded by $\mathcal{A}_1,...,\mathcal{A}_{28}$. Let $\theta = \theta(r_1,r_2,d)$ be the sum of inner dihedral angles at each edge of $P$. We have $\theta(\frac{1}{2},\frac{1}{2},8)=0$ and $\theta(\frac{3}{4},\frac{3}{4},8)>2\pi$. So we can choose $r_1,r_2,d$ such that $\theta = 2\pi$.

\begin{figure}[h]
\centering
\includegraphics[width=12cm]{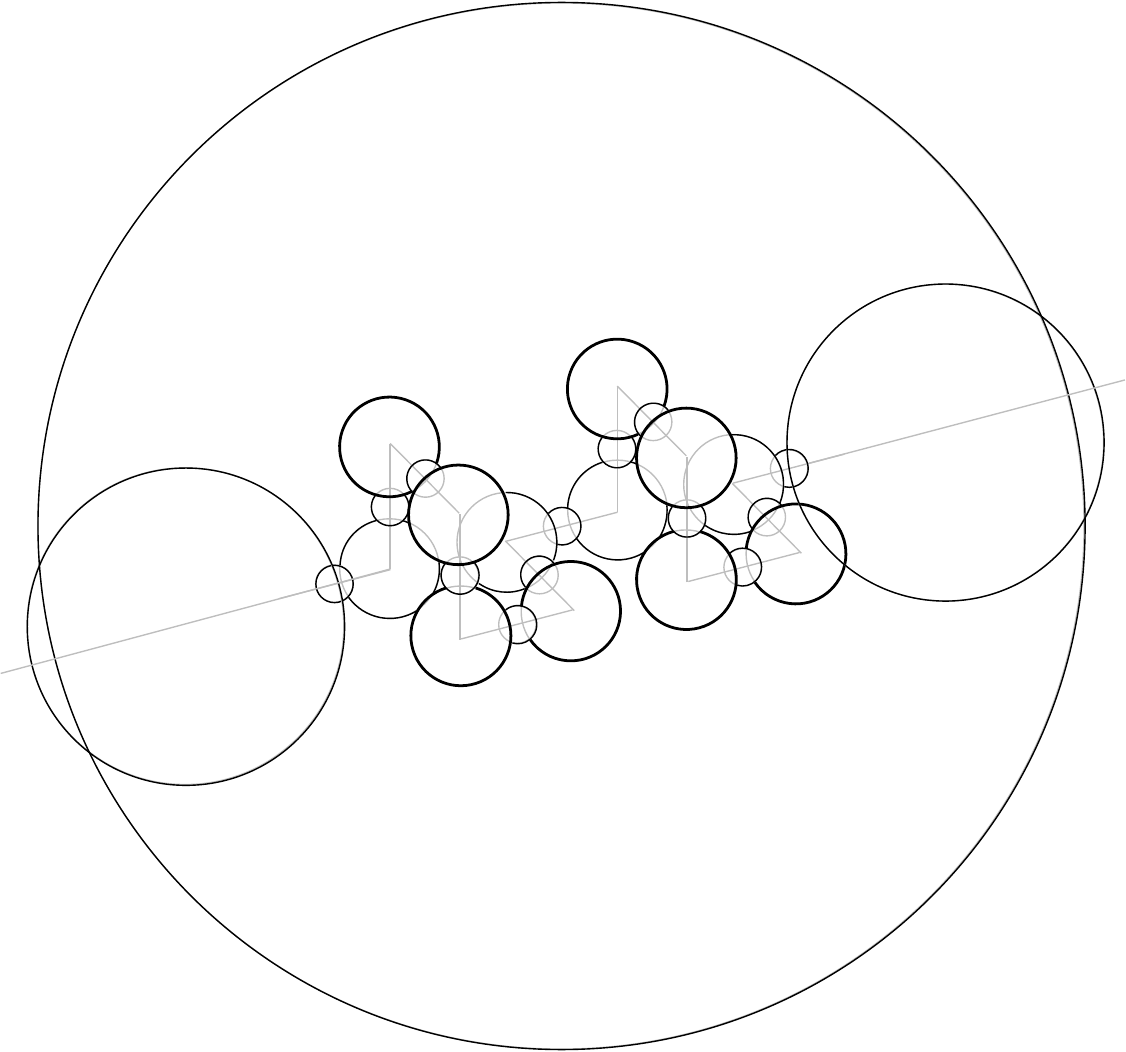}
\caption{Example 3 illustration.}
\label{fig:Example3}
\end{figure}

Let $p_1\in E_1$ such that the $z$-coordinate of $p_1$ is $r$. That is, $p_1$ is the point in $E_1$ with maximum $z$-coordinate. Given $p_i\in E_i$, choose $p_{i+1}\in E_{i+1}$ such that the winding number $w(\gamma_i)=0$ for any circle arc $\gamma_i\subset\mathcal{A}_i$ connecting $p_i$ to $p_{i+1}$. So we can choose circle arc $\gamma_i$ connecting $p_i$ to $p_{i+1}$ so that it is orthogonal to both $E_i, E_{i+1}$. We can check that $\gamma = \gamma_1 ...\gamma_{28}$ form a closed loop and $[\gamma]$ generates $\pi_1(P)$. 

We now have marked annuli $(\mathcal{A}_i, \gamma_i)$ for $i\in\mathbb{Z}/28\mathbb{Z}$. There are unique (loxodromic) M\"obius transformations identifying these marked faces in standard identification pattern, that is, $(\mathcal{A}_1\stackrel{A_1}{\longrightarrow} \mathcal{A}_3), (\mathcal{A}_4\stackrel{B_1}{\longrightarrow} \mathcal{A}_2), (\mathcal{A}_5\stackrel{A_2}{\longrightarrow} \mathcal{A}_7), (\mathcal{A}_8\stackrel{B_2}{\longrightarrow} \mathcal{A}_6), $ etc.

We can check that the edges $E_1\stackrel{A_1}{\longrightarrow} E_4\stackrel{B_1}{\longrightarrow} E_3 \stackrel{A_1^{-1}}{\longrightarrow} E_2 \stackrel{B_1^{-1}}{\longrightarrow} E_5 \stackrel{A_2}{\longrightarrow} ... \stackrel{B_7^{-1}}{\longrightarrow} E_1$ form a \emph{geometric cycle of edges} as defined in \cite{AG11}, and by the main theorem in the same paper we have the face-pairing transformations generate a surface group $\Gamma$ where $P$ is a fundamental domain for $\Gamma$. (In order to check the \emph{geometric cycle of edges} condition we can construct two other piecewise-circle-arc loops around $\partial P$ compatible with the determined side identification maps, this shows that the return map fixes 3 points on $E_1$ which must be singular-elliptic with rotation angle $\theta=2\pi$ which must then be the identity.) Moreover, by theorem \ref{theoremGammaEuler} we get $(S^3 - \Lambda_\Gamma )/ \Gamma$ is a conformally flat circle bundle with Euler number 1 over a surface of genus 7.

In this example we have $\Gamma = \langle A_1, B_1,..., A_7,B_7\ |\ \prod[A_i,B_i] = 1 \rangle$ where $A_i$ are loxodromics without rotation. This allows for more freedom of deformations as will be described below.

\subsection{Deformations of surface groups}
\label{SecDeform}
Let $G$ be a Lie group. For any $A\in G$ we define $C_G(A)$ be the centralizer of $A$ in $G$.  Note that $C_G(A)\subset G$ is a Lie subgroup which contains 1-parameter subgroups through $A$ (assuming $A$ is in the identity component of $G$).

\bigskip

Let $\Gamma\subset G$ be a surface group with standard generators $$\Gamma = \langle A_1, B_1,..., A_g,B_g\ |\ W(A_1, B_1,..., A_g,B_g)= \prod[A_i,B_i] = 1 \rangle.$$
We will now describe algebraic deformations of surface group $\Gamma$ which corresponds to earthquake/grafting in the level of representations.

\bigskip

\noindent {\bf Non-separating simple closed loops.}
Let $C\in C_G(A_1)$ not the identity, and let $B_1' = B_1 C$. Then $A_1B_1'A_1^{-1}B_1'^{-1} =A_1B_1CA_1^{-1}C^{-1}B_1^{-1} = A_1B_1A_1^{-1}B_1^{-1}. $ Then 
$$
\Gamma' = \langle A_1, B'_1,A_2,B_2..., A_g,B_g\ | W(A_1, B_1',...,A_g,B_g) = 1 \rangle
$$
is another surface group in $G$. In most cases, $\Gamma'$ is not a conjugate of $\Gamma$.

\noindent {\bf Separating simple closed loops.}
Let $k$ be an integer between $1$ and $g-1$. (Think of $k$ as the genus of a component of $\Sigma_g$ with a separating simple closed loop removed.)  Let $C\in C_G(\prod_{i=1}^k [A_i,B_i])$. For $i=1,...,k$, let $A'_i = CA_i C^{-1}$ and $B'_i = CB_i C^{-1}$. Let $A'_i = A_i$ and $B'_i = B_i$ for $i=k+1,...,g$. We can check that $\prod_{i=1}^k [A'_i,B'_i]) = \prod_{i=1}^k [A_i,B_i])$ and thus
$$
\Gamma' = \langle A'_1, B'_1,..., A'_g,B'_g\ |  \prod[A'_i,B'_i] = 1 \rangle
$$
is another surface group.

\begin{definition}
We call the above operation an \emph{algebraic earthquake} on surface group representations.
\end{definition}

\emph{Example 4.}
Let $\rho:\pi_1(\Sigma_g)\rightarrow G = \Isom^+(\mathbb{H}^3) \cong SO^+(3,1)$ be a Fuchsian representation and let $\Gamma = \rho(\pi_1(\Sigma_g))$ with standard generators. We have $\Gamma$ is purely loxodromic preserving a totally geodesic $\mathbb{H}^2\hookrightarrow\mathbb{H}^3$, so in particular $A_1$ is loxodromic without rotation. We have $C_G(A_1) \cong \mathbb{R}_+ \oplus SO(2)$ where the $\mathbb{R}_+$ factor corresponds to the 1-parameter group of non-rotating loxodromics containing $A_1$, and the $SO(2)$ factor correponds to the elliptic elements  having the same fixed points (2 points in $S^2$) as $A_1$. Indeed deforming $\Gamma$ using the $\mathbb{R}_+$ factor corresponds to an earthquake, and using the $SO(2)$ factor corresponds to a grafting along the free homotopy class of $\rho^{-1}(A_1)$ which is represented by a non-separating simple closed loop on the surface. The same analogy works for the case of earthquake/grafting along a separating simple closed loop $\rho^{-1}(\prod_{i=1}^k[A_i,B_i])$.

\bigskip

\noindent {\it Example 3 (continued).} 
In the example constructed, $A_i$ are non-rotating loxodromic generators. In particular, $\mathcal{A}_{25}\stackrel{A_7}{\longrightarrow}\mathcal{A}_{27}$ is non-rotating loxodromic. So upto a conjugation of the whole surface group, $A_7$ acts as a scaling on $\mathbb{R}^3$: $A_7\vec{x} = \lambda\vec{x}$. Let $P'$ be the image of $P$ under this conjugation, so $\mathcal{A}'_{25}$ and $\mathcal{A}'_{27}$ are ``concentric''. So $C_G(A_i) \cong \mathbb{R}_+\oplus SO(3)$. Consider an algebraic earthquake using the $SO(3)$ part of $C_G(A_7)$: we compose $B_7$ with such a rotation. For small rotation angles, the deformation of $\Gamma$ can be realized geometrically as a deformation of the fundamental domain $P'$: rotating $\mathcal{A}'_{26}$ by Euclidean rotations fixing the Euclidean center of $\mathcal{A}'_{25} $. We can rotate $\mathcal{A}_{26}$ until it is tangent to another face, this represent a path in the space of representations to a group with accidental parabolic. Proceeding ``past'' this point of degeneration gives us non-discrete surface group.

\section{Bounding the Euler number}
\label{secEulernumber}

The goal of this section is to bound the Euler number of a circle bundle which admits flat conformal structures. The first approach is by using a fundamental domain and applying theorem \ref{theoremGammaEuler}. The second approach relies on the formulation of the Euler number of a disc bundle over a surface as the self-intersection number of a section which may be constructed to be piecewise geodesic. 

\subsection{Fundamental domain approach}

\begin{definition}
Let $X$ be a metric space. For $x\in X, \varepsilon>0$, we denote $B_\varepsilon(x)= \{y\in X\ |\ d(x,y)<\varepsilon \}$ which we call a \emph{ball} of radius $\varepsilon$ centered at $x$. For a subset $S\subset X$, we denote $N_\varepsilon(S) = \bigcup_{x\in S} B_\varepsilon(x)$ which is called the $\varepsilon$-neighborhood of $S$.
\end{definition}
We now prove the following:
\begin{theorem}
With the same hypothesis as in theorem \ref{theoremGammaEuler}, we get $|e(E)|<\frac{3}{2}n^2$ where $n$ is the number of faces of the fundamental polyhedron.
\end{theorem}
\begin{proof}

Let $\gamma$ be the piecewise-circle-arc loop on $\partial P$ as constructed in theorem \ref{theoremGammaEuler}. The strategy is to construct a nearby piecewise-circle-arc loop $\beta \in S^3 - P$, so that the Euler number $e(E)$ can be computed (up to a sign) as a linking number $lk(\gamma,\beta)$.

First, let's construct $\beta$. As before, let $\mathcal{A}_i,...,\mathcal{A}_n$ be the faces of $P$, and let $E_1,...,E_n$ be the edges of $P$ such that $\mathcal{A}_i$ contains $E_i, E_{i+1}$ for $i\in\mathbb{Z}/n\mathbb{Z}$. Let $B_i$ be the bisecting 2-sphere between $\mathcal{A}_i$ and $\mathcal{A}_{i-1}$, note that $B_i$ contains $E_i$.

\begin{figure}[h]
\centering
\includegraphics[width=7cm]{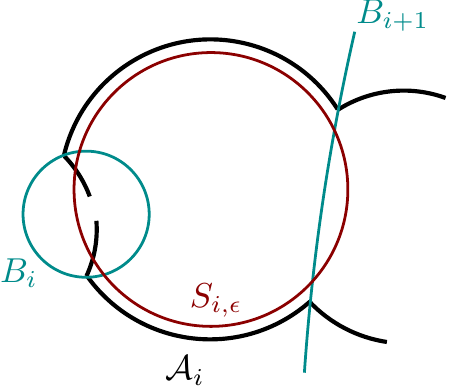}
\caption{Illustration of the thickened annulus ${\bf A_i }$}
\label{fig:ThickenedFace}
\end{figure}

Let $S_i$ be the 2-sphere containing the M\"obius annulus $\mathcal{A}_i$. We know that each 
$S_i = \{x\in S^3\ |\ d_S(x,x_i) = r_i \}$, a metric 2-sphere, where $x_i\in S^3, r_i\in\mathbb{R}_{>0}$ and $d_S$ is the standard spherical metric. Also, $S_i = \exp \left(S(0,r_i)\right) $ where $S(0,r_i)$ is a sphere of radius $r_i$ centered at $0$ in  $T_{x_i}S^3$. If $p\in int(\mathcal{A}_i)\subset S_i$ such that $p=\exp(v)$ for $v\in T_{x_i}S^3$, we have for small $\epsilon>0$, $\exp((1+\epsilon)v)$  is either in $int(P)$ or not in $P$. If $\exp((1+\epsilon)v)\in int(P)$, we let $S_{i,\epsilon} =  \exp \left(S(0,(1-\epsilon)r_i)\subset T_{x_i}S^3\right)$ a sphere slightly ``smaller" than $S_i$. Otherwise, we let $S_{i,\epsilon} =  \exp \left(S(0,(1+\epsilon)r_i)\subset T_{x_i}S^3\right)$. We choose $\epsilon$ small enough so that $S_{i,\epsilon}$ intersects $B_i$ and $B_{i+1}$ transversely.

For each annulus $\mathcal{A}_i$, there is an associated thickened annulus ${\bf A_i }$ in $S^3 - P$. More specifically, ${\bf A_i}$ is a polyhedron bounded by $S_i, S_{i,\epsilon}, B_i, B_{i+1}$. Since $P$ has finitely many faces, we can choose $\epsilon$ small enough so that the interior $int({\bf A_i })$ is contained in  $S^3 - P$ for all $i$. Let $F_{i,0}$ and $F_{i,1}$ be the boundary pieces of ${\bf A_i }$ which are contained the the spheres $B_i, B_{i+1}$. These are closed annuli. Let $int(F_{i,0}), int(F_{i,1})$ be the interior of these annuli as subset of $B_i, B_{i+1}$ respectively. The polyhedron ${\bf A_i }$ is a thickened M\"obius annulus, it has the property that for any two points $q_i\in int( F_{i,0})$ and $q_{i+1}\in int(F_{i,1})$ can be connected by a circle arc lying completely in $int({\bf A_i })$. 

For $i\in\mathbb{Z}/n\mathbb{Z}$ we choose $q_i\in  int( F_{i,0}) \cap  int( F_{i-1,1})$. We can connect $q_i, q_{i+1}$ by a circle arc $\beta_i$ such that $\beta_i\subset int({\bf A_i })\subset S^3 - P$. Thus $\beta =\beta_1,...,\beta_n$ is a loop composed of $n$ circle arcs. Moreover, $\beta$ generates $H_1(S^3 -P)$.

By theorem \ref{theoremGammaEuler} we have $e(E) = [\gamma]\in \pi_1(P)\cong H_1(P)$. Consider $S^3 - \beta \supset P$, there is a deformation retract $S^3 - \beta$ to $P$, thus $H_1(S^3 - \beta)\cong H_1(P)$ and we have $e(E) = \pm [\gamma] \in H_1(S^3 - \beta)$. Therefore $e(E) = \pm lk(\gamma,\beta)$.

Let us consider now the linking number computation by a planar link diagram which can be obtained by projecting $\gamma,\beta$ to a generic plane in $\mathbb{R}^3\subset S^3$. Each circle arc belongs to some circle; a pair of circles is projected to a pair of ellipses. Consider the crossings between these two ellipses, there are at most three $(+)$ crossings and at most three $(-)$ crossings. Both $\gamma,\beta$ are composed of $n$-circle arcs, thus in the link diagram of $\gamma,\beta$, there are at most $\frac{3}{2}n^2$ crossings having the same sign. So $e(E)\le \frac{3}{2}n^2$. Equality cannot be achieved since we can always find segments $\gamma_i$ and $\beta_j$ and a generic projection so that their projected images cross no more than twice. Therefore we have  $e(E)<\frac{3}{2}n^2$.
\end{proof}

We will now provide comments on the possibility of a linear bound. Note that in the construction of $\beta$ above, we can choose its vertices $q_1,...,q_n$ arbitrarily close to the vertices $p_1,...,p_n$ of $\gamma$. Consider $\gamma,\beta$ as loops in $\mathbb{R}^3$. Replacing each circle arc segment of $\gamma$ and $\beta$ with a straight segment we obtain (Euclidean) polygonal unknots $\gamma'$ and $\beta'$ with the property that $|lk(\gamma',\beta') - lk(\gamma,\beta)|\le 2n$. So in the interest of establishing a linear bound for $lk(\gamma,\beta)$, we can work with piecewise linear unknots instead. To show a linear bound on $e(E)$, it suffices to prove the following:
\begin{conjecture}
\label{conjectureLinking}
Let $\gamma$ be a piecewise linear unknot with vertices $p_i,...,p_n$. Let $\varepsilon>0$ be small enough such that $N_\varepsilon(\gamma)$ is a tubular neighborhood. Then there is a constant $c$ such that for every choice of $q_i\in B_\varepsilon(p_i)$, the piecewise linear unknot $\beta$ constructed by connecting $q_1,...,q_n$ (in order) has the property: $lk(\gamma,\beta)<cn$.
\end{conjecture}

\subsection{Self-intersection number approach}
This subsection present an independent approach to bounding the Euler number of a (4-dimensional) hyperbolic disc bundle over a closed surface. Of course these manifolds are closely related to conformally flat circle bundles considered earlier. Indeed, quasi-Fuchsian convex cocompact surface groups $\Gamma\subset SO(4,1)$ produce both a hyperbolic disc bundle and a conformally flat circle bundle as quotient geometric manifolds.

Slightly changing the notation we let $E$ be a disc bundle over a closed surface $\Sigma_g$ with a complete hyperbolic structure $E = \mathbb{H}^4 / \Gamma$. The Euler number of this bundle, denoted by $e(E)$, is then the self-intersection number of a section $\Sigma_g\stackrel{s_0}{\hookrightarrow} E$ which can be grasped geometrically. We can combinatorially construct a section as follows:

By the $(\textnormal{dev},\rho)$ pair we can identify $\tilde{E}\cong \mathbb{H}^4$ and $\pi_1(\Sigma_g)\cong\tilde\Gamma$. We choose a standard generator set 
$$\Gamma =\langle A_1, B_1,..., A_g, B_g\ |\ \prod[A_i,B_i] = 1 \rangle.  $$
\begin{definition}
\label{partialWords}
We define the \emph{partial words} of  $\prod_{i=1}^g[A_i,B_i]$ to be $W_1=A_1, W_2=A_1B_1, W_3 = A_1B_1A_1^{-1}, ..., W_{4g} = \prod[A_i,B_i] = 1$.
\end{definition}

\begin{figure}[h]
\centering
\includegraphics[width=5.5cm]{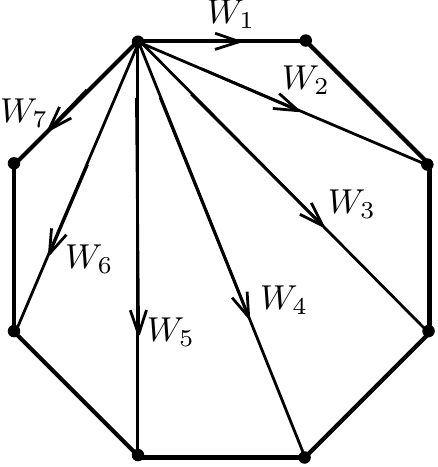}
\caption{Combinatorial picture of the partial words.}
\label{fig:PartialWords}
\end{figure}

Pick a point $x_0\in E$ and a representative $\tilde{x_0} \in \mathbb{H}^4$. Let $W_1,...,W_{4g}$ be the partial words of $\prod[A_i,B_i]$, we can connect the points $W_1 \tilde{x_0}, W_2 \tilde{x_0}, ..., W_{4g} \tilde{x_0} = \tilde{x_0}$ by geodesics in $\mathbb{H}^4$ that form a triangulation of a $4g$-gon into $4g-2$ triangles. Moreover there is a unique geodesic plane containing given 3 points in $\mathbb{H}^4$, so we have $4g-2$ geodesic triangles with vertices at $W_1 \tilde{x_0}, W_2 \tilde{x_0}, ..., W_{4g} \tilde{x_0}$. Under $\mathbb{H}^4 \rightarrow \mathbb{H}^4/\Gamma$ these triangles descend down to geodesic triangles in $E$ based at $x_0$ which continuously form a surface. Let $s_{x_0}:\Sigma\rightarrow E$ be a continuous (piecewise immersion) map whose image consist of these triangles.
\begin{definition}
Given a choice of vertex $x_0\in E$, the piecewise immersion $s_{x_0}: \Sigma_g \rightarrow E$ constructed above is well-defined up to homeomorphisms on $\Sigma_g$ homotopic to the identity. We call this a \emph{standard $4g-2$ polygonal map} of surface into $E$. We let $\tilde{s}_{\tilde{x}_0}: \widetilde{\Sigma}_g \rightarrow \mathbb{H}^4$ be the $\Gamma$-equivariant lift based at $\tilde{x_0}$.
\end{definition}

By construction $s_{x_0}$ induces an isomorphism $\pi_1(\Sigma_g)\stackrel{s^*_{x_0}}{\rightarrow}\pi_1(E)$, so it is homotopic to a section of the disc bundle $E$. We don't yet know for sure if $s_{x_0}$ is regularly homotopic (homotopic through immersions) to a section. If we know this then we can safely compute the Euler number of $E$ as the self-intersection of $s_{x_0}$. One possible avenue to show this is to apply Hirsch's theorem in \cite{Hirsch59Immersions} which guarantees the existence of such a regular homotopy.

For now we will assume that $s_{x_0}$ is an embedding which means it is a section. An approach to bound $e(E)$ suggested to the author by Feng Luo is as follows. The section described above is determined by the choice of an initial point $\tilde{x_0}\in\mathbb{H}^4$. We can move $\tilde{x_0}$ slightly and get a different section and hope that the two sections only intersect transversely, and since each geodesic triangle transervely intersects another at no more than one point, we would have a rough bound $e(E)< (4g-2)(36g -23)$. This number comes from the $4g -2$ geodesic triangles in $\mathbb{H}^4$, and each triangle may possibly intersect $36g-23$ triangles around it if its vertices in $\mathbb{H}^4$ are perturbed.

We observe that there are $4$ dimensions of freedom in choosing the base point which determines the standard $4g-2$ polygonal section. However, the space of all geodesic planes intersecting a given plane non-transversely has dimension $\ge 4$. So a non-trivial problem here is to show the existence of another standard $4g-2$ polygonal section transverse to the original one.

\begin{definition}
Let $A\in\Isom^+(\mathbb{H}^4) =\Mob^+(S^3)$ be a rotating loxodromic transformation with rotation angle not $2k\pi$ for some integer $k$. Then there is a circle $C$ in $S^3$ containing both fixed points and $C$ is invariant under $A$. We call $C$ the \emph{rotation axis} of $A$. Also, $C$ bounds a complete totally geodesic 2-dimensional plane $H\subset \mathbb{H}^4$. Depending on the context, we also say $H$ is the rotation axis of $A$.
\end{definition}

\begin{lemma}
\label{lemmaRotatingLoxo}
Let $A\in\Isom^+(\mathbb{H}^4) =\Mob^+(S^3)$ be a rotating loxodromic transformation with rotation angle not $k\pi$ for some integer $k$. Then there is a unique $3$-dimensional totally geodesic subspace $\mathbb{H}^3\hookrightarrow\mathbb{H}^4$ that is $A$-invariant.
\end{lemma}
\begin{proof}
For all $x\in\mathbb{R}^3 = S^3 - \{\infty\}$, we have (up to a conjugation) $A(x) = \lambda R(x)$ for some $\lambda\in\mathbb{R}^+$ and $R\in SO(3)-\{ I\}$. Any 3-dimensional subspace invariant under $A$ corresponds to a 2-sphere in $S^3$ invariant under $A$. This invariant 2-sphere must contain the fixed points: $0,\infty$, so it must be a Euclidean plane through $0$ in $\mathbb{R}^3$. There's only one such plane invariant under $R$ which is the one orthogonal to its rotation axis. Thus there is a unique 3-dimensional totally geodesic subspace invariant under $A$.
\end{proof}
Note that the above is not true if the rotation angle of $R$ is $k\pi$ for any interger $k$. If the rotation angle is $\pi$, then $A$ leaves invariant any 2-sphere in $S^3$ that contains the rotation axis (which is a circle).

\begin{definition}
Let $S_1,...,S_k$ be subsets of $\mathbb{H}^n$. We define $\textnormal{span}(S_1,...,S_k)$ to be the smallest totally geodesic subspace $\mathbb{H}^m\hookrightarrow\mathbb{H}^n$ that contains $S_1,...,S_k$. 
\end{definition}

\begin{lemma}
\label{lemmaDisjointTransverse}
Let $H_1, H_2$ be two geodesic planes in $\mathbb{H}^4$. If span$(H_1, H_2) = \mathbb{H}^4$ then $H_1, H_2$ are either disjoint or intersecting transversely.
\end{lemma}
\begin{proof}If $H_1, H_2$ intersect non-transversely (along a plane or a geodesic), then span$(H_1, H_2)$ is at most 3-dimensional.  \end{proof}

\begin{lemma}
\label{lemmaDeformX0}
Suppose we have a surface group $\Gamma\subset G = \Isom^+(\mathbb{H}^4)$, and a choice of generators and partial words as in Definition \ref{partialWords} which are all loxodromic. Let $A,B,P,Q$ be such loxodromic transformations. Then there is a point $x_0'$ arbitrarily close to $x_0$ and $x_1'$ arbitrarily close to $x_1$ so that $\textnormal{span}(x'_0, Ax'_0, Bx'_0)$ is either disjoint from $ \textnormal{span}(x'_1, Px'_1, Qx'_1)$ or intersecting $ \textnormal{span}(x'_1, Px'_1, Qx'_1)$ transversely.
\end{lemma}
\begin{proof}
 Let $\Delta(x_0,Ax_0,Bx_0)$ be the geodesic triangle with vertices $x_0,Ax_0,Bx_0$. The context is that this is a triangle in $\mathbb{H}^4$ which descends to one of the $4g-2$ geodesic triangles immersed in $E$.
Let $H =  \textnormal{span}(x_1, Px_1, Qx_1)$ and let  $K = \textnormal{span}(H, x_0, Ax_0, Bx_0)$. If $K$ is 4-dimensional then we're done: $x'_0 = x_0$. If $K$ is 2-dimensional, then $H = \textnormal{span}(x_0, Ax_0, Bx_0)$. Then we can move $x_0$ an arbitrarily small amount to $x'_0$ outside of $H$ and the problem is reduced to the case when $K$ is 3-dimensional. This is when $H$ and span$(x_0, Ax_0, Bx_0)$ intersect along a geodesic.

Case 1: Suppose $K$ is not invariant under $A$. Then there is a vector $v\in T_{x_0}K$ so that $dA(v)\not\in T_{Ax_0} K$. The condition $dA(v)\not\in T_{Ax_0}K$ is an open condition, so we have the freedom to choose $v$ such that $v\not\in T_{x_0} H$. Then there is $x'_0$ arbitrarily close to $x_0$ along the $v$ direction where we have $x'_0\in K - H$ and $Ax'_0\not\in K$. So $\textnormal{span}(H,x'_0) = K$, thus $\textnormal{span}(H,x'_0, Ax'_0) = \mathbb{H}^4$. Therefore $\textnormal{span}(x'_0, Ax'_0, Bx'_0)$ and $H$ are either disjoint or intersecting transversely.

A similar proof works for the case $K$ is not invariant under either $B$ or $AB^{-1}$.

Case 2: From now we assume that $K$ is invariant under all three transformations $A, B,$ and $AB^{-1}$. Since $K$ is 3-dimensional, it divides $\mathbb{H}^4$ into two half-spaces.

Case 2a: If $A$ is rotating with angle not $k\pi$ for an interger $k$, then by lemma there is a unique 3-dimensional geodesic subspace $S_A$ invariant under $A$. We deform $x_0$ to $x'_0$ outside of both $K$ and $S_A$. Thus $K' = \textnormal{span}(H, x'_0, Ax'_0, Bx'_0)$ is either 4-dimesional (in which case we're done by lemma \ref{lemmaDisjointTransverse}), or a 3-dimensional space distinct from $K$. In the latter case, $A$ cannot preserve $K'$ as well by lemma \ref{lemmaRotatingLoxo}, so we are back in Case 1 which has already been resolved.

The above argument works if either $A, B,$ or $AB^{-1}$  is rotating with angle not $k\pi$.

Case 2b: Now suppose that $A,B, AB^{-1}$ are all rotating with angle $k\pi$ for some integer $k$.

 If $A,B,AB^{-1}$ are all non-rotating, then their action on $K$ is orientation preserving. Pick a vector $v\in T_{x_0}\mathbb{H}^4$ that is orthogonal to $K$. Its images under the differential maps are $dA(v)\in T_{Ax_0}\mathbb{H}^4$ and $dB(v)\in T_{Bx_0}\mathbb{H}^4$. So $v, dA(v), dB(v)$ must point into the same half space of $\mathbb{H}^4 - K$ because $A,B$ are both orientation preserving. We can then find $x'_0$ arbitrarily close to $x_0$ along the direction of $v$, and we have $x'_0, Ax'_0, Bx'_0$ are all in the same half space, and also the distance from $x'_0, Ax'_0, Bx'_0$ to $K$ are the same. Thus $\textnormal{span}(x'_0, Ax'_0, Bx'_0)$ is disjoint from $K$, in particular $\textnormal{span}(x'_0, Ax'_0, Bx'_0)$ is disjoint from $H$.

The last case is when two out of the three transformations $A,B,AB^{-1}$ are rotating with angle $\pi$, the other is non-rotating. Let $C_1,C_2$ be the rotation axes of these two loxodromic with $\pi$-rotation, these are circles in $\partial_\infty \mathbb{H}^4$. Since $K$ is invariant under $A,B,AB^{-1}$ (by Case 2 assumption), we have $\partial_\infty K \cong S^2$ contains the two rotation axes. We can move $x_0$ to $x'_0$ outside of $K$. So we have $K' = \textnormal{span}(H, x'_0, Ax'_0, Bx'_0)$, and $K\cap K' = H$, thus $\partial_\infty K \cap \partial_\infty K' = \partial_\infty H$. If  $\partial_\infty K'$ does not contain the two rotation axes then $K'$ is not invariant under all three transformations $A,B,AB^{-1}$ and we are reduced to Case 1. Otherwise, both $\partial_\infty K, \partial_\infty K'$ contain the two rotation axes, which means the two axes are the same circle and the same as $\partial_\infty H = \partial_\infty K \cap \partial_\infty K'$. Now recall that $H = \textnormal{span}(x_1, Px_1, Qx_1)$. We can move $x_1$ an arbitrarily small amount to $x_1'$ so that $H' = \textnormal{span}(x'_1, Px'_1, Qx'_1)$ is not the rotation axis of either $A$ or $B$ or $AB^{-1}$. Apply the same argument we get $\textnormal{span}(x'_0, Ax'_0, Bx'_0)$ is either disjoint from $H'$ or intersecting $H'$ transversely.

This concludes the proof of lemma \ref{lemmaDeformX0}

\end{proof}

\begin{theorem}
Let $E$ be a disc bundle over a closed surface $\Sigma_g$ with a complete uniformizable hyperbolic structure $E=\mathbb{H}^4 /\Gamma$ where $\Gamma$ is a surface group with a loxodromic standard generator set $A_i, B_i$ and loxodromic partial words $W_1=A_1, W_2=A_1B_1, W_3 = A_1B_1A_1^{-1}, ..., W_{4g} = \prod[A_i,B_i] = 1$. Suppose also that there is a point $x_1\in E$ so that the induced standard $4g-2$ polygonal map $\Sigma_g\stackrel{s_{x_1}}{\rightarrow} E$ is injective, hence a section of $E$.

Then $$e(E)\le(4g-2)(36g-23).$$
\end{theorem}

\begin{proof}
Let $x_1\in E$ and a lift $\tilde{x}_1\in\mathbb{H}^4$. As before we can construct $4g-2$ geodesic triangles in $\mathbb{H}^4$ with vertices at  $W_1 \tilde{x_1}, W_2 \tilde{x_1}, ..., W_{4g} \tilde{x_1}$ which descends to a continuous map $s_{x_1}:\Sigma_g\rightarrow E$ which is assumed to be a section. We can choose $x_1$ such that no lift $\tilde{x}_1$ is in the rotation axis of any loxodromic element of $\Gamma$. (We can do this since $\Gamma(\bigcup\textnormal{all rotation axes})$ has measure $0$ in $\mathbb{H}^4$.)  The Euler number is the the self-intersection number$s_{x_1}$. The $4g-2$ geodesic triangle under $\Gamma$ expand to a countable collection of geodesic triangles which we name $\{T_{x_1,i} \}_{i\in\mathbb{Z}^+}$.

We pick $\tilde{x}_0$ arbitraily close to $\tilde{x}_1$ in a way that $\tilde{x}_0$ is not in $T_{1,i}$ for any $i$. For every $\tilde{x}'_0$ in the ball $B(\tilde{x}_0,\epsilon)\subset\mathbb{H}^4$ we can once again construct $\{T_{\tilde{x}'_0,i} \}_{i\in\mathbb{Z}^+}$ a countable collection of geodesic triangles from $\tilde{x}'_0$ and $\Gamma$. Let $s_{\tilde{x}'_0}: \Sigma_g\rightarrow E$ be the piecewise geodesic section based at $\tilde{x}'_0$. We have $\tilde{x}'_0$ is close to $\tilde{x}_0$ which is close to $\tilde{x}_1$, so the sections $s_{x_1}$ and $s_{\tilde{x}'_0}$ are arbitrarily close. If we can make the two sections transverse (by choosing $\tilde{x}'_0)$), then $e(E)$ is the local intersection number between them.

Let $S_{i,j} \subset B(\tilde{x}_0,\epsilon)$ be the set of points $\tilde{x}'_0$ near $\tilde{x}_0$ such that $T_{\tilde{x}'_0, j}$ intersects the fixed triangle $T_{x_1,i}$ non-transversely. By lemma \ref{lemmaDeformX0}, the set $S_{i,j}$ has empty interior. Moreover $S_{i,j}$ is a closed set since disjoint/transverse is an open condition. Therefore by Baire Category theorem, $$S = \bigcup_{i,j\in\mathbb{Z}^+} S_{i,j}$$ has empty interior. So we can choose $\tilde{x}'_0 \in B(\tilde{x}_0,\epsilon) - S$ and we have $T_{x_1,i}$ is disjoint or transverse to $T_{\tilde{x}'_0,j}$ for any $i,j\in\mathbb{Z}^+$. Thus we have two arbitrarily close transverse sections $s_{x_1}$ and $s_{\tilde{x}'_0}$ and therefore $$e(E) = i(s_{x_1},s_{\tilde{x}'_0}) \le (4g-2)(36g-23).$$

\end{proof}

\bibliographystyle{unsrt} 
\bibliography{SonsBib}

\begin{thebibliography}{10}

\bibitem{GLT}
M.~Gromov, H.~Lawson, and W.~Thurston.
\newblock Hyperbolic 4-manifolds and conformally flat 3-manifolds.
\newblock {\em Publ. Math. of IHES}, 68:27--45, 1988.

\bibitem{Kuiper88}
N.~Kuiper.
\newblock Hyperbolic 4-manifolds and tesselations.
\newblock {\em Math. Publ. of IHES}, 68:47--76, 1988.

\bibitem{Kapovich89Flat}
M.~Kapovich.
\newblock Flat conformal structures on three-dimensional manifolds: the
  existence problem. i.
\newblock {\em Sibirsk. Mat. Zh.}, 30:60--73, 1989.

\bibitem{Goldman83Conf}
W.~Goldman.
\newblock Conformally flat manifolds with nilpotent holonomy.
\newblock {\em Trans. Amer. Math.Soc.}, pages 573--583, 1983.

\bibitem{Kapovich07Klein}
M.~Kapovich.
\newblock Kleinian groups in higher dimensions.
\newblock {\em Progress in Mathematics}, 265:485--562, 2007.

\bibitem{Ratcliffe94Book}
J.~Ratcliffe.
\newblock {\em Foundations of hyperbolic manifolds}.
\newblock Springer, 1994.

\bibitem{Kapovich01Book}
M.~Kapovich.
\newblock {\em Hyperbolic manifolds and discrete groups}.
\newblock Birkhauser Boston Inc., Boston MA, 2001.

\bibitem{ThurstonNotes}
W.~Thurston.
\newblock The geometry and topology of three-manifolds.
\newblock Lecture notes, 1980.

\bibitem{AG11}
S.~Anan'in and C.~H. Grossi.
\newblock Yet another poincare polyhedron theorem.
\newblock {\em Proceedings of the Edinburgh Mathematical Society (Series 2)},
  54(02):297--308, 2011.

\bibitem{Hirsch59Immersions}
M.~W. Hirsch.
\newblock Immersions of manifolds.
\newblock {\em Transactions of the AMS}, 93:242--276, 1959.

\end{thebibliography}

\end{document}